\documentclass{amsart}

\usepackage{amsmath,leftidx,braket}
\usepackage[matrix,arrow]{xy}
\usepackage[draft]{fixme}
\usepackage[psamsfonts]{amssymb}
\usepackage[utf8]{inputenc}
\usepackage{graphicx}

\usepackage[unicode]{hyperref}

\newcounter{thm}
\newtheorem{proposition}[thm]{Proposition}
\newtheorem{lemma}[thm]{Lemma}
\newtheorem*{maintheorem}{Main Theorem}
\newtheorem{theorem}[thm]{Theorem}
\newtheorem{corollary}[thm]{Corollary}
\newtheorem{remark}[thm]{Remark}
\theoremstyle{definition}

\newtheorem*{definition*}{Definition}

\newcounter{step}[section]
\newcommand{\step}{\refstepcounter{step}\par\medskip\noindent{\bfseries Step \thestep.\ }}

\renewcommand{\Re}{\mathop{\mathrm{Re}}}
\renewcommand{\Im}{\mathop{\mathrm{Im}}}
\newcommand{\eps}{\varepsilon}

\newcommand{\bbC}{\mathbb C}
\newcommand{\bbD}{\mathbb D}
\newcommand{\bbH}{\mathbb H}

\newcommand{\bbR}{\mathbb R}
\newcommand{\bbZ}{\mathbb Z}

\newcommand{\cal}{\mathcal}

\newcommand{\CZ}{{\bbC/\bbZ}}
\newcommand{\RZ}{{\bbR/\bbZ}}
\newcommand{\CZbar}{\overline{\CZ}}
\newcommand{\HZ}{{\bbH/\bbZ}}
\newcommand{\HZbar}{\overline{\HZ}}
\newcommand{\HpmZ}{{\bbH^\pm/\bbZ}}
\newcommand{\HpmZbar}{\overline{\HpmZ}}
\newcommand{\HpZ}{{\bbH^+/\bbZ}}
\newcommand{\HpZbar}{\overline{\HpZ}}
\newcommand{\HmZ}{{\bbH^-/\bbZ}}
\newcommand{\HmZbar}{\overline{\HmZ}}

\def\i{{\rm i}}
\newcommand{\dfn}{{:=}}
\newcommand{\dx}{{\ {\rm d}}}

\newcommand{\quadand}{{\quad\text{and}\quad}}

\DeclareMathOperator{\rot}{rot}
\DeclareMathOperator{\Per}{Per}
\DeclareMathOperator{\distance}{dist}
\DeclareMathOperator{\distortion}{dis}

\newcommand{\E}{\mathfrak E}

\title{Complex rotation numbers}

\begin{author}{Xavier Buff}
\email{xavier.buff@math.univ-toulouse.fr}
\address{ %
  Institut de Math\'ematiques de Toulouse\\
  Universit\'e Paul Sabatier\\
  118, route de Narbonne \\
  31062 Toulouse Cedex \\
  France }
\end{author}

\begin{author}{Nataliya Goncharuk}
\email{natalka@mccme.ru}
\address{%
National Research University Higher School of Economics \\
Russia, Moscow, Miasnitskaya Street 20\\
and Independent University of Moscow\\
Russia, Moscow, Bolshoy Vlasyevskiy Pereulok 11}
\end{author}

\thanks{The research of the first author was supported by the IUF. The research of the second author was supported by the following grants: RFBR projects 10-01-00739-a and 12-01-33020 mol$\_$a$\_$ved, joint RFBR/CNRS project 10-01-93115-CNRS$\_$a, Moebius Contest Foundation for Young Scientists, Simons Foundation.}
\begin{document}

\begin{abstract}
We investigate the notion of complex rotation number which was introduced by V.\,I.\,Arnold in 1978. Let $f\colon\RZ\to \RZ$ be an orientation preserving circle diffeomorphism and let $\omega\in \CZ$ be a parameter with positive imaginary part. Construct a complex torus by glueing the two boundary components of the annulus $\Set{z\in \CZ|0<\Im(z)<\Im(\omega)}$ via the map $f+\omega$. This complex torus is isomorphic to $\bbC/(\bbZ+\tau\bbZ)$ for some appropriate $\tau\in \CZ$.

According to Moldavskis \cite{M}, if the ordinary rotation number $\rot(f+\omega_0)$ is Diophantine and if $\omega$ tends to $\omega_0$ non tangentially to the real axis, then $\tau$ tends to $\rot(f+\omega_0)$. We show that the Diophantine and non tangential assumptions are unnecessary: if $\rot(f+\omega_0)$ is irrational then $\tau$ tends to $\rot(f+\omega_0)$ as $\omega$ tends to $\omega_0$.

This, together with results of N.Goncharuk \cite{NG}, motivates us to introduce a new fractal set (``bubbles''), given by the limit values of $\tau$ as $\omega$ tends to the real axis. For the rational values of $\rot (f+\omega_0)$, these limits do not necessarily coincide with $\rot (f+\omega_0)$ and form a countable number of analytic loops in the upper half-plane.
\end{abstract}

\maketitle

Notation:
\begin{itemize}
\item $\bbH=\bbH^+$ is the set of complex numbers with positive imaginary part.
\item $\bbH^-$ is the set of complex numbers with negative imaginary part.
\item If $p/q$ is a rational number, then $p$ and $q$ are assumed to be coprime.
\item If $x$ and $y$ are distinct points in $\RZ$, then $(x,y)$ denotes the set of points $z\in \RZ-\set{x,y}$ such that the three points $x,z,y$ are in increasing order and
$[x,y] \dfn (x,y)\cup \set{x,y}$.
\item $\rot (f) \in \RZ$ is a rotation number of an orientation-preserving circle diffeomorphism $f$.
\item If $f\colon\RZ\to \RZ$ is a circle diffeomorphism, $\displaystyle D_f \dfn  \int_{\RZ} \left|\frac{f''(x)}{f'(x)}\right| \dx x.$
\end{itemize}

\section*{Introduction}

Given an orientation preserving analytic circle diffeomorphism $f\colon\RZ\to \RZ$ and a parameter $\omega\in \HZ$, set
\[f_\omega\dfn f+\omega\colon\RZ\to \RZ+\omega.\]
The circles $\RZ$ and $\RZ+\omega$ bound an annulus $A_\omega\subset \CZ$. Glueing the two sides of $A_\omega$ via $f_\omega$, we obtain a complex torus $E(f_\omega)$, which
may be uniformized as $\cal{E}_\tau \dfn \bbC/(\bbZ+\tau\bbZ)$ for some appropriate $\tau\in \HZ$, the homotopy class of $\RZ$ in $E(f_\omega)$ corresponding
to the homotopy class of $\RZ$ in $\cal{E}_\tau$.
The \emph{complex rotation number} of $f_\omega$ is $\tau_f(\omega) \dfn \tau$. It is the complex analogue of the ordinary rotation number of $f+t$ for $t\in \RZ$.

V.\,I.\,Arnold's problem \cite{Arn}, generalized by R. Fedorov and E. Risler independently, is to study the relation of the ordinary rotation number of the circle diffeomorphism
$f\colon\RZ\to \RZ$ and the limit behaviour of the complex rotation number $\tau_f(\omega)$ as $\omega$ tends to $0$.

According to work of Risler \cite[Chapter 2, Proposition 2]{Ris}, the function
\[\tau_f\colon\HZ\to  \HZ\] is holomorphic.
We shall show that there is a continuous extension of $\tau_f$ to
\[\HZbar \dfn \HZ\cup \RZ.\]

The ordinary rotation number of a circle homeomorphism $f\colon\RZ\to \RZ$ is defined as follows.
Let $F\colon\bbR\to \bbR$ be a lift of $f\colon\RZ\to \RZ$. Such a lift is unique up to addition of an integer. The sequence of functions $\frac{1}{n}\bigl(F^{\circ n}-{\rm id}\bigr)$
converges uniformly to a constant function $\Theta$. If we replace $F$ by $F+k$ with $k\in \bbZ$, the limit $\Theta$ is replaced by $\Theta+k$, so that the value $\rot(f)\in
\RZ$ of $\Theta$ modulo $1$ only depends on $f$. This is the rotation number of $f$. Note that the rotation number is rational if and only if the circle homeomorphism has a
periodic orbit.

Our main result, proved in Section \ref{sec_contin},  concerns the behavior of $\tau_f(\omega)$ as $\omega$ tends to $\RZ$.
Recall that a periodic orbit of a circle diffeomorphism is called \emph{parabolic} if its multiplier is 1, and it is called \emph{hyperbolic} otherwise. A circle
diffeomorphism with periodic orbits is called \emph{hyperbolic} if it has only hyperbolic periodic orbits.

\begin{maintheorem}
Let $f\colon\RZ\to \RZ$ be an orientation preserving analytic circle diffeomorphism. Then, the  function $\tau_f\colon\HZ\to \HZ$ has a continuous extension $\bar\tau_f\colon\HZbar\to \HZbar$. Assume $\omega\in \RZ$.
\begin{itemize}
\item If $\rot(f_\omega)$ is irrational, then $\bar \tau_f(\omega)=\rot(f_\omega)$.
\item If $\rot(f_\omega)=p/q$ is rational, then $\bar \tau_f(\omega)$ belongs to the closed disk of radius $D_f/(4\pi q^2)$ tangent to $\RZ$ at $p/q$; moreover
\begin{itemize}
\item if $f_\omega$ has a parabolic cycle, then $\bar \tau_f(\omega)=\rot(f_\omega)$.
\item if $f_\omega$ is hyperbolic, then $\bar \tau_f(\omega) \in \HZ$, in particular $\bar \tau_f(\omega)\neq
\rot(f_\omega)$.
\end{itemize}
\end{itemize}
\end{maintheorem}

Our main contribution to this result is the case of irrational (yet not Diophantine) rotation number, and the continuous extension of $\tau_f$ to the whole boundary $\RZ$.
The particular case of this theorem:
\begin{corollary}
If $\rot (f)$ is irrational, then  $\tau_f(\omega) $ converges to $\rot (f)$ when $\omega$ goes to zero.
\end{corollary}
solves the problem posed by {\'E}tienne Ghys (however he refers to V.Arnold) \cite[p. 25]{Ghyslist}.

The case of Diophantine rotation numbers was investigated earlier by E.Risler \cite[Chapter 2]{Ris} and V.Moldavskis \cite{M} independently. Risler constructed the map $\tau_f$ in a some subset $\Omega_s$ of $\HZ$; $\Omega_s$ is detached from points $\omega \in \RZ$ with $\rot (f_{\omega}) \in \mathbb Q/ \mathbb Z$. He also studied the behavior of $\tau_f$ and obtained some formulas and estimates on its derivatives; in particular, he proved that  $\tau_f$ is injective on $\Omega_s$ provided that $f$ is close to rotation.

The case of parabolic cycles was studied by J.Lacroix (unpublished) and N.Goncharuk \cite{NG} independently. The case of hyperbolic diffeomorphisms was dealt first by Yu. Ilyashenko and V. Moldavskis \cite{YuI_M}, then this result was improved by N.Goncharuk \cite{NG}. For exact statements of these results, see Section \ref{sec-proof}.

In Appendix \ref{sec_inf}, we shall also study the behavior of $\tau_f(\omega)$ as the imaginary part of $\omega$ tends to $+\infty$.

\section{Bubbles: a new fractal set}

The Main Theorem enables us to define a new interesting fractal set, related to the circle diffeomorphism, namely the set $\bar \tau_f(\RZ)$. Due to the Main Theorem, this set contains $\RZ$ and a countable number of loops --- “bubbles”, the endpoints of bubbles are rational points of $\RZ$ (see the sketch in Fig. \ref{fig-bubbles}). Due to Theorem \ref{th_mine}, these loops are analytic curves.
\begin{figure}
\includegraphics[width=0.8 \textwidth]{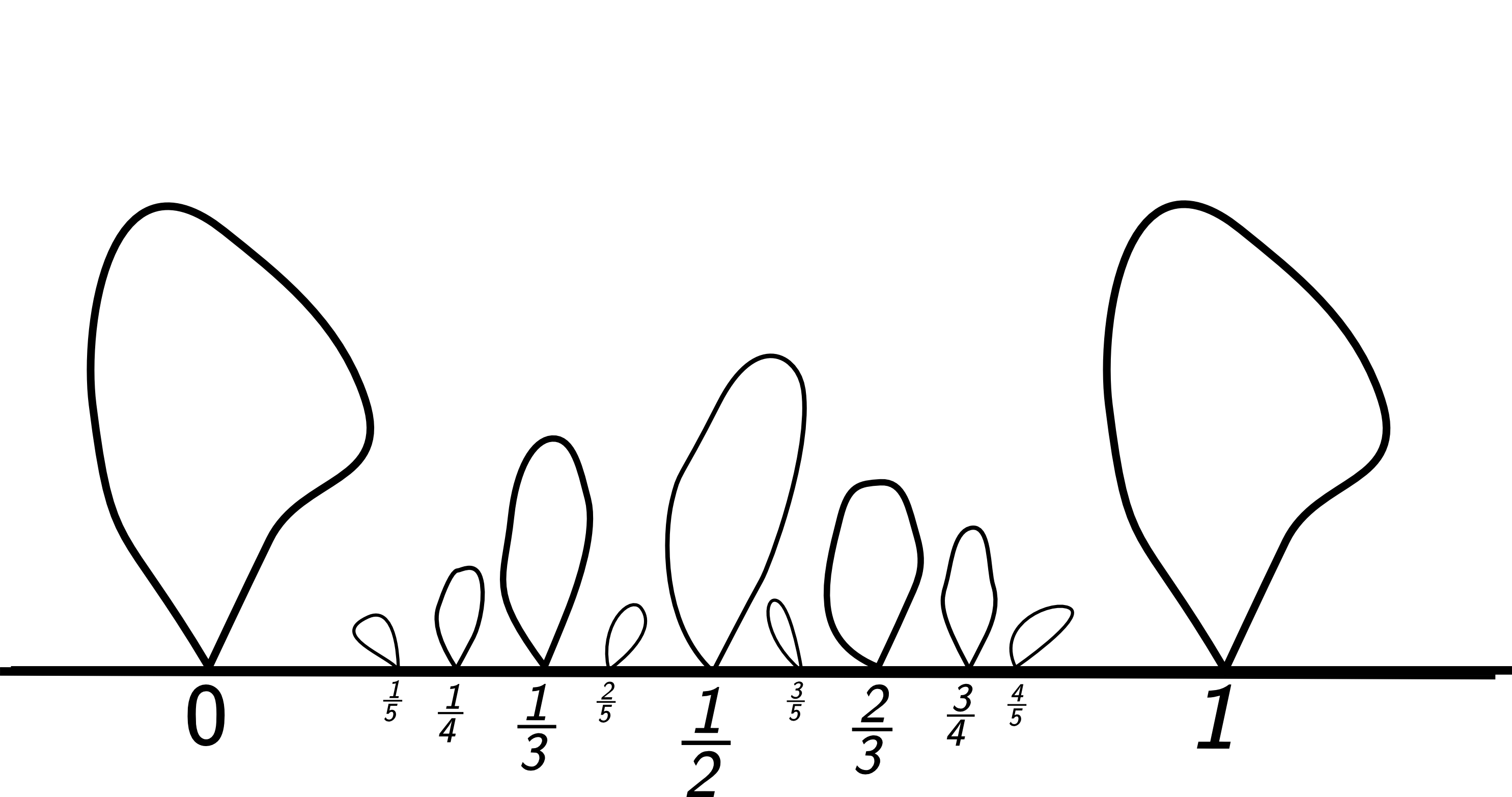}
\caption{Bubbles. The sketch of the set $\bar \tau_f(\RZ)$.}
\label{fig-bubbles}
\end{figure}

There are many natural questions about the geometrical structure of the set $\bar \tau_f(\RZ)$:
\begin{enumerate}
	\item \label{it:univ}Is it true that $\bar \tau_f(\RZ)$ is the boundary of $\tau_f(\HZ)$, and $\tau_f$ is univalent?
	\item \label{it:size}How large are bubbles?
	\item \label{it:inter}Do different bubbles intersect each other?
	\item \label{it:shape}What is the shape of a bubble? In particular, could a bubble be self-intersecting?
	\item \label{it:bundle} What can be said about the shape of a ``bubble bundle'', when several bubbles grow from the same point of the real axis (see Fig. \ref{fig-bundle})?
\end{enumerate}

We disprove the conjecture of item (\ref{it:univ}), see Corollary \ref{coroll-noninj} at the end of this Section.

As for  item (\ref{it:size}), the following lemma is a part of Main Theorem:

\begin{lemma} (Size of bubbles)
 \label{lemma-bubblesize}
 The bubble corresponding to $\rot (f_{\omega})=p/q$ belongs to the disk tangent to $\RZ$ at $p/q$ with radius $C/q^2$, where $C = D_f /(4\pi)$. If $f$ is $C^1-$close to a rotation, then  $C$ is close to $0$.
 \end{lemma}
 This implies that when $f$ is close to a rotation, different bubbles do not intersect (item (\ref{it:inter})).

  The question on the shape of bubbles (item (\ref{it:shape})) is still open, however our results clarify the shape of bubbles near their endpoints. Let us introduce the following classification:

  \begin{definition*}
If all maps $f_{\omega}, \omega \in (\omega_0,\omega_0+\eps]$ are hyperbolic, and $f_{\omega_0}$ is not, then $\omega_0$ is called \emph{a (left) endpoint of a bubble}. In this case, $\bar \tau_f(\omega) \to \rot (f_{\omega_0})$ as $\omega \to \omega_0, \omega>\omega_0$, due to the continuity of $\bar \tau_f(\omega)$.

If the multiplier of some fixed point of $f_{\omega}$ tends to one as $\omega \to \omega_0, \omega>\omega_0$, then $\omega_0$ is called \emph{a real (left) endpoint} of a bubble. For example, this happens if some parabolic cycle of $f_{\omega_0}$ bifurcates into real hyperbolic cycles as $\omega$ increases.

If the multipliers of  fixed points of $f_{\omega}$ do not tend to one as $\omega \to \omega_0, \omega>\omega_0$, then $\omega_0$ is called \emph{a complex (left) endpoint} of a bubble. This means that all parabolic cycles of $f_{\omega_0}$ bifurcate into complex conjugate cycles as $\omega$ increases. Note that in this case, $f_{\omega_0}$ must have other hyperbolic cycles, otherwise $f_{\omega}, \omega \in (\omega_0,\omega_0+\eps]$ cannot be hyperbolic.

In an analogous way, we introduce the notion of right endpoints of bubbles.
\end{definition*}

\begin{lemma}(Real endpoints)
\label{lemma-realendpoints}
 If $\omega_0$ is a real endpoint of the bubble, $\rot (f_{\omega_0})=p/q$, then the curve $\bar \tau_f(\omega), \, \omega \to \omega_0, \omega>\omega_0$, tends to $p/q \in \RZ$ from above: enters any horocycle at the point $p/q$, see Fig. \ref{fig-realcomplex} (a).
\end{lemma}
\begin{lemma}(Complex endpoints)
\label{lemma-complexendpoints}
 If  $\omega_0$ is a complex endpoint of the bubble, $\rot (f_{\omega_0})=p/q$, then the curve $\bar \tau_f(\omega), \, \omega \to \omega_0, \omega>\omega_0$, is located between two horocycles at $p/q$.

 For the left endpoint, this curve is  tangent to the segment $[p/q, p/q+\eps)$, see Fig. \ref{fig-realcomplex} (b).
 For the right endpoint, it will be tangent to the segment $(p/q-\eps, p/q]$.
\end{lemma}

\begin{figure}[h]
\includegraphics[width=0.3\textwidth]{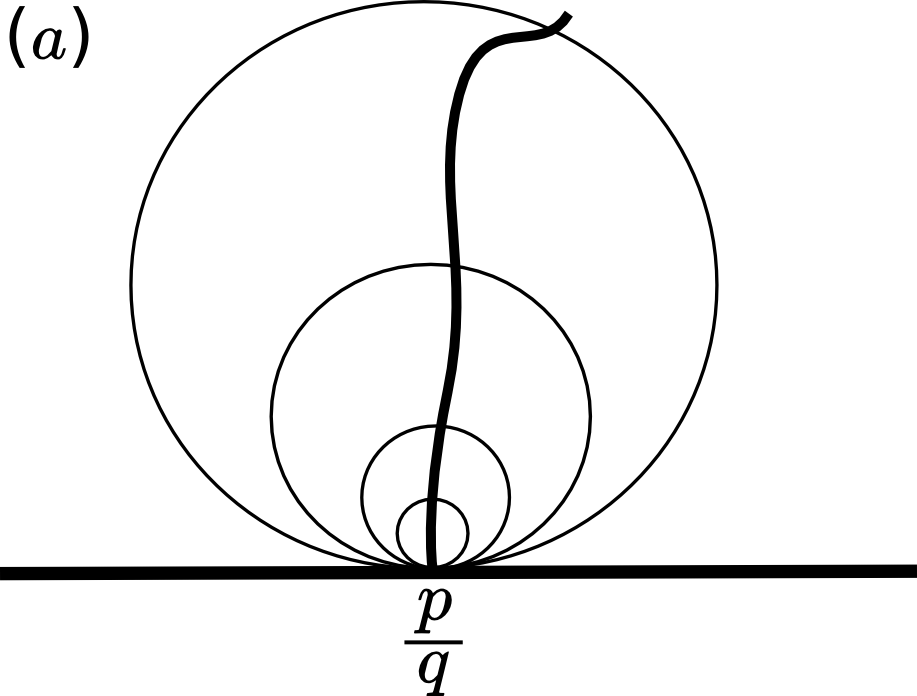}
\hfil
\includegraphics[width=0.3\textwidth]{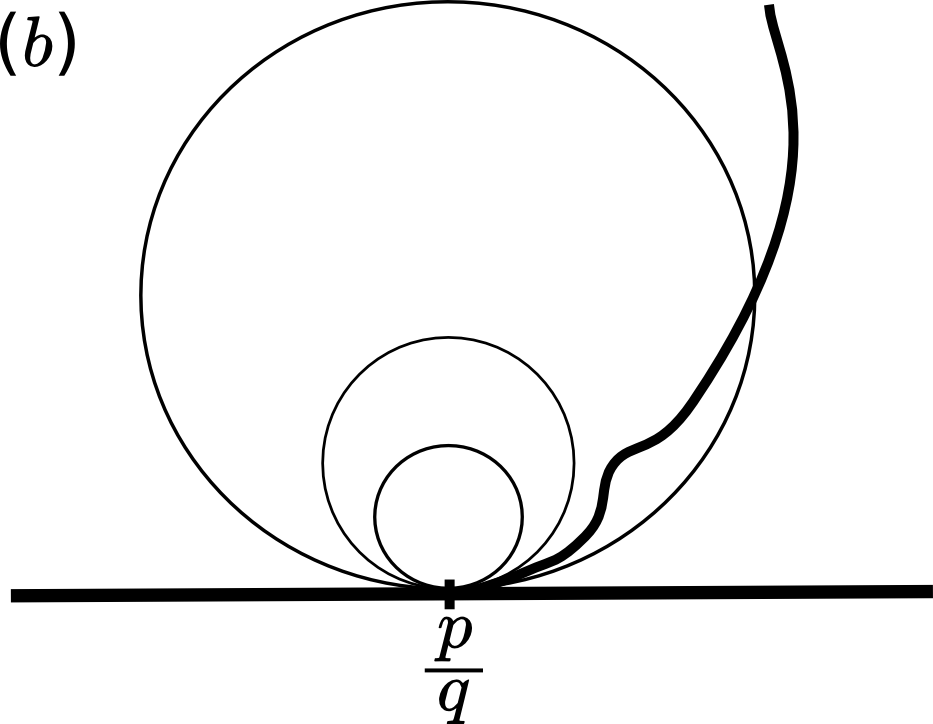}
\caption{The sketch of the curve $\bar \tau_f(\omega)$, $\omega \in [\omega_0, \omega_0+\eps]$, in the case when $\omega_0$ is  (a) real and  (b) complex left endpoint of the bubble; $\rot (f_{\omega_0}) = \frac{p}{q}$.}\label{fig-realcomplex}
\end{figure}

\begin{remark}
\label{remark-leftRight}
 When we pass to  the  diffeomorphism $x\mapsto -f(-x)$, the map $\bar \tau_f$ is conjugated by $z \mapsto -\bar z$. Thus it is sufficient to prove Lemmas \ref{lemma-realendpoints} and \ref{lemma-complexendpoints} only for left endpoints.
\end{remark}

We finish this section by Corollary \ref{coroll-noninj} which disproves the conjecture of item (\ref{it:univ}).

\begin{corollary}
\label{coroll-noninj}
Assume that $x-f(x)$  has two local maxima at points $x_1$ and $x_2$ with  $x_1-f(x_1) \neq x_2-f(x_2)$ (see Fig. \ref{fig-graphs}). Then, $\tau_f$ is not injective.
\end{corollary}

\begin{proof}
\begin{figure}[h]
\includegraphics[width=\textwidth]{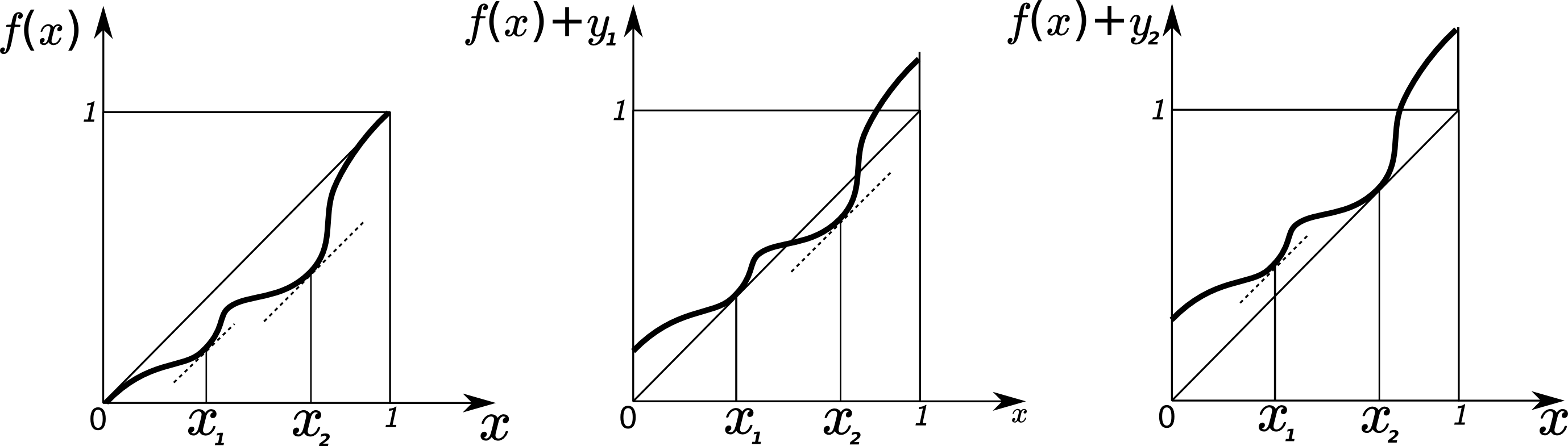}
\caption{Graphs of the functions $f, f+y_1, f+y_2$.}
\label{fig-graphs}
\end{figure}

Let $y_1$ and $y_2$ be the respective values of $x-f(x)$ at $x_1$ and $x_2$.
Suppose that $y_1<y_2$. Then the map $f_\omega$ for $y_1<\omega<y_2$ has zero rotation number, and it has parabolic fixed points for $\omega =y_1$ and $ \omega=y_2$. Note that when $\omega$ increases from $y_1 $ to $y_1+\eps$, the parabolic fixed point disappears ($y_1$ is a complex left endpoint of a bubble), thus due to Lemma \ref{lemma-complexendpoints}, the curve $\omega \mapsto \bar \tau_f(\omega)$
is tangent to $[0, 0+\eps)$. When $\omega<y_2$ tends to $y_2$, the two hyperbolic fixed points merge into a parabolic fixed point ($y_2$ is a real endpoint of a bubble). Thus, according to Lemma \ref{lemma-realendpoints},
the curve $\omega \mapsto \bar \tau_f(\omega)$ enters any horocycle at $0$ as $\omega<y_2$ tends to $y_2$. Fig. \ref{fig-oneBubble} shows the sketch of this bubble. But if $\tau_f$ were injective, the pair of germs of the curve $\bar \tau_f |_{\RZ}$
at $y_1$ and $y_2$ (both passing through $0$) would be oriented clockwise. The contradiction shows that $\tau_f$ is not injective in the upper half-plane.
\end{proof}

\begin{figure}[h]
\includegraphics[width=0.3\textwidth]{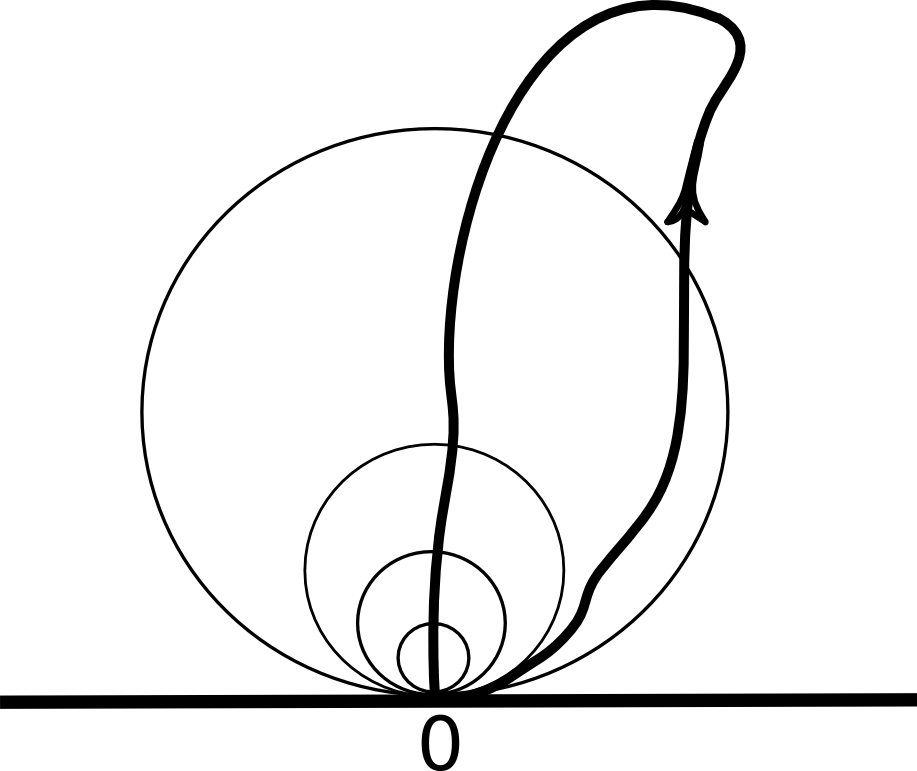}
\caption{The sketch of the curve $\bar \tau_f((y_1, y_2))$.}
\label{fig-oneBubble}
\end{figure}

For the map $f$ whose graph is shown in Fig. \ref{fig-graphs}, the same arguments give some information about the curve $\bar \tau_f((0, y_2))$ --- the bubble bundle  (see item (\ref{it:bundle})). However we cannot choose between numerous possible pictures, see Fig. \ref{fig-bundle} for two of them. The question on the exact shape of a bubble bundle stays open.

\begin{figure}[h]
\includegraphics[width=0.65\textwidth]{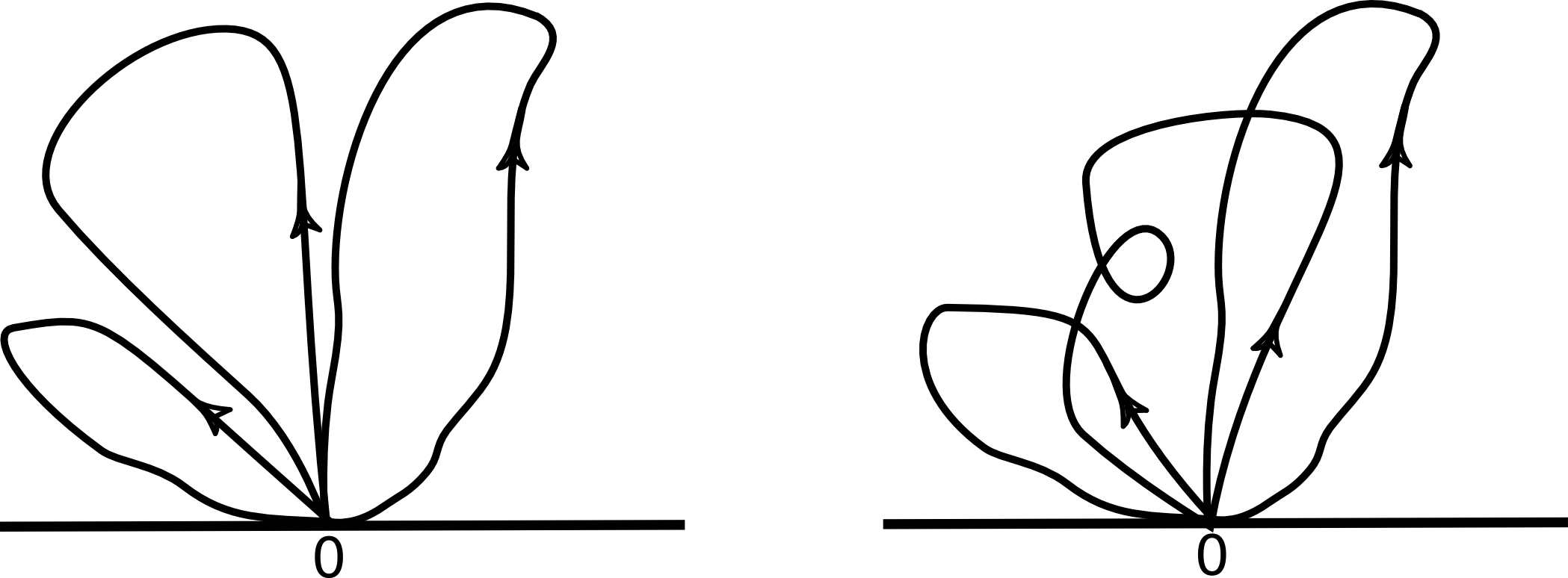}
\caption{The ``bubble bundle'': the possible sketches of the curve $\bar \tau_f((0, y_2))$.}
\label{fig-bundle}
\end{figure}

\section{Strategy of the proof}
\label{sec-proof}

The proof of the Main Theorem goes as follows.

\setcounter{step}{0}
\step
Recall that a number $\theta\in \RZ$ is {\em Diophantine} if there are constants $c>0$ and $\beta>0$ such that for all rational numbers $p/q\in \mathbb Q/\bbZ$, we have
\[ \left|x - \frac{p}{q}\right| > \frac{c}{q^{2+\beta}}.\]

\begin{theorem}[V. Moldavskis \cite {M}]
\label{th_Mold}
If $\omega\in \RZ$ and if $\rot(f_\omega)$ is Diophantine,  then
\[
\lim_{\substack{y \to 0\\ y>0}} \tau_f(\omega+\i y) = \rot(f_\omega).
\]
\end{theorem}

 Theorem \ref{th_Mold} was proved by Risler \cite{Ris} as a corollary of a very delicate analog of the Arnold--Hermann theorem. A very short direct proof was obtained by Moldavskis \cite {M}.

\step\label{step_rational}
If $\omega\in \RZ$ and $\rot(f_\omega)$ is rational, then the conclusion of Theorem \ref{th_Mold} is not true. This fact was first proved by Yu. Ilyashenko and V. Moldavkis
\cite{YuI_M}. We do not formulate their result since we will use its later generalized version.

\begin{theorem}[N. Goncharuk \cite{NG}]
\label{th_mine}
If $\omega\in \RZ$, if $\rot(f_\omega)$ is rational and if $f_\omega$ is hyperbolic, then $\tau_f$ extends analytically to a neighborhood of $\omega$.
\end{theorem}

In the following, we shall denote by $\bar\tau_f(\omega)$ this extension of $\tau_f$ at $\omega$. At this stage, $\bar \tau_f(\omega)$ is defined on a countable number of real segments. However, in what follows we will define $\bar \tau_f$ on the whole $\RZ$.

\step
Recall that $\theta\in \RZ$ is {\em Liouville} if it is irrational but not Diophantine. We use the following result of Tsujii.

\begin{theorem}[M. Tsujii \cite{Tsujii}]
\label{th_Tsujii}
The set of $\omega\in \RZ$ such that $\rot(f_\omega)$ is Liouville has zero Lebesgue measure.
\end{theorem}
It implies that almost every $\omega \in \bbR/\bbZ$ satisfies assumptions of either Theorem \ref{th_Mold}, or Theorem \ref{th_mine} (note that the set of $\omega$ such that $f_\omega$ has a parabolic cycle is countable, because our family is analytic).

\step
If $f_\omega$ has rational rotation number, we usually denote it by  $p/q$. We denote by $\Per(f_\omega)$ the set of periodic points of $f_\omega\colon\RZ\to \RZ$. For $x\in \Per(f_\omega)$, we denote by $\rho_x$ the multiplier of $x$ as a fixed point of $f^{\circ q}$. Our contribution starts with the following result. It is an analog of the Yoccoz Inequality which bounds the multiplier of a fixed point of a polynomial in terms of its combinatorial rotation number \cite{H}.

\begin{lemma}\label{lemma_sizebubble}
Assume that $f_\omega$ is a hyperbolic map with rational rotation number $p/q$. Then, $\bar \tau_f(\omega)$ belongs to the disk tangent to $\RZ$ at $p/q$ with radius
\begin{equation}
\label{eq:Romega}
R_\omega\dfn \frac{1}{2\displaystyle \pi q\cdot \sum_{x\in \Per(f_\omega)} \frac{1}{|\log \rho_x|}}.
\end{equation}
In addition,
\begin{equation}
\label{eq:Bubblesize}R_\omega\leq D_f/(4\pi q^2).
\end{equation}

\end{lemma}

The cardinal of $\Per(f_\omega )$ for a hyperbolic map is at least $2q$, and according to Lemma \ref{th_as-Denjoy}, for each $x\in \Per(f_\omega)$ we have $|\log \rho_x|\leq D_f$. Thus the estimate \eqref{eq:Romega} yields \eqref{eq:Bubblesize}.

Estimate \eqref{eq:Bubblesize} immediately implies Lemma \ref{lemma-bubblesize}.

Note that Lemma \ref{lemma_sizebubble} implies Lemma \ref{lemma-realendpoints}. Indeed, for real endpoints of bubbles, one of the multipliers $\rho_x$ tends to 1, and \eqref{eq:Romega} yields that $R_\omega$ tends to $0$ as $\omega \to 0$. Thus $\bar \tau_f(\omega)$ enters any horocycle at $p/q$.

\step
Let $\bar\tau_f\colon\RZ\to \CZ$ be defined by
\begin{itemize}
\item $\bar\tau_f(\omega) \dfn \rot(f_\omega)$ if the rotation number of $f_\omega$ is irrational or if $f_\omega$ has a parabolic cycle and
\item $\displaystyle \bar \tau_f(\omega)\dfn \lim_{\substack{y \to 0\\y>0}} \tau_f(\omega+\i y)$ if $f_\omega$ is hyperbolic.
\end{itemize}
This definition agrees with the definition of $\bar \tau_f(\omega)$ for hyperbolic $f_{\omega}$ (see \autoref{step_rational}). We are going to prove that $\bar \tau_f$ is the continuous extension of $\tau_f$ to the real axis; so the coincidence of the notation with that of Main Theorem is not accidental and will not lead to confusion.
\begin{lemma}\label{lem-continuity}(Continuity of the boundary function)
The function $\bar \tau_f$ is continuous on $\RZ$.
\end{lemma}

It is particularly difficult to prove the continuity of $\bar \tau_f$ at complex endpoints of bubbles. For the points $\omega$ where $f_{\omega}$ is hyperbolic (points of bubbles), it follows from Theorem \ref{th_mine}; for real endpoints of bubbles, we use  Lemma \ref{lemma-realendpoints}; for the points with irrational $\rot(f_{\omega})$, we need Lemma \ref{lemma-bubblesize}.

\step
The holomorphic map $\tau_f\colon\HZ\to \HZ$ has radial limits on $\RZ$ almost everywhere, and those limits coincide with the continuous map $\bar\tau_f$. It follows easily
that $\tau_f$ extends continuously by $\bar\tau_f$ to $\RZ$.

\section{Multipliers of periodic orbits and distortion}

Before embarking into the proof of our results, we shall obtain the useful estimate on multipliers of periodic orbits of a circle diffeomorphism (Lemma \ref{th_as-Denjoy}).

The \emph{distortion} of a diffeomorphism $f\colon I\to J$ is
\[\distortion_{I}  (f) = \max_{x,y \in I} \log  \frac{f'(x)}{f'(y)}.\]
If $f\colon I\to J$ and $g\colon J\to K$ are diffeomorphisms, then
\[\distortion_J (f^{-1}) = \distortion_I (f)\quad \text{and}\quad \distortion_I ( g \circ f) \le \distortion_I (f) + \distortion_J (g).\]

\begin{lemma}[Denjoy]
\label{denjoy1}
Let $f\colon\RZ\to \RZ$ be an orientation preserving diffeomorphism and  $I\subset \RZ$ be an interval such that $I$, $f(I)$, $f^{\circ 2}(I)$, \ldots, $f^{\circ n}(I)$ are disjoint. Then,
\[\distortion_{I}  (f^{\circ n}) \leq D_f.\]
\end {lemma}

\begin{proof}
Let $x$ and $y$ be points in $I$. Set $x_k\dfn f^{\circ k}(x)$ and $y_k \dfn f^{\circ k}(y)$.
Then,
\begin{align*}
\bigl|\log (f^{\circ n})'(x) - \log (f^{\circ n})'(y) \bigr|& =
\left|\sum_{k=0}^{n-1} \log f'(x_k)-\log f'(y_k)\right|
\\
& \leq  \sum_{k=0}^{n-1} \left|\int_{x_k}^{y_k} \frac{f''(x)}{f'(x)}\dx x\right| \leq \int_\RZ \left| \frac{f''(x)}{f'(x)}\right|\dx x = D_f.\qedhere
\end{align*}
\end{proof}

As a corollary, we have the following control on  the multipliers of the periodic orbits of $f$. This result is surely known by specialists, but we include its proof due to the lack of a suitable reference.

\begin{lemma}(Estimate on multipliers)
\label{th_as-Denjoy}
Let $f\colon\RZ\to \RZ$ be an orientation preserving diffeomorphism and $\rho$ be the multiplier of a cycle of $f$. Then, $|\log \rho| \leq D_f$.
\end{lemma}

\begin{proof}
The average of the derivative $(f^{\circ q})'$ along the circle $\RZ$ is equal to $1$. As a consequence, there exists a point $x_0
\in \RZ$ such that $(f^{\circ q})'(x_0)=1$. Any periodic orbit $\set{x, f(x), \dots, f^{\circ q}(x)=x}$ divides the circle into disjoint
intervals $I_1,\dots, I_q$ which are permuted by $f$. Without loss of generality, we may assume that $I_1$ contains $x$ and $x_0$. Then, according to the previous Lemma,
\[|\log \rho| = \bigl|\log (f^{\circ q})'(x)\bigr|=\left|\log \frac{(f^{\circ q})'(x)}{(f^{\circ q})'(x_0)}\right|\leq \distortion_{I_1} (f^{\circ q}) \leq D_f.\qedhere
\]
\end{proof}

\section{The Diophantine case}

We include a proof of Theorem \ref{th_Mold}. It is a simplified version of original proof of Moldavskis \cite{M}.

The proof relies on the following lemma on quasiconformal maps which is classical.

\begin{lemma}\label{lem-quasi-dist}
Suppose that there exists a $K$-quasiconformal map between two complex tori $E_1$ and $E_2$. Then
\[\distance_{\bbH} (\tau(E_1), \tau(E_2)) \le \log K\]
where $\distance_{\bbH}$ is the hyperbolic distance in $\bbH$, and where $\tau(E_{1})\in \bbH$ and $\tau(E_2)\in \bbH$ are moduli with respect to corresponding generators in $H_1(E_1)$ and $H_1(E_2)$.
\end{lemma}

Without loss of generality, we may assume that $\omega=0$, so $f\colon\RZ\to \RZ$ has Diophantine rotation number $\theta\in \RZ$. A theorem of Yoccoz (see \cite{Yoc}) asserts that there is an
analytic circle diffeomorphism $\phi\colon\RZ\to \RZ$ conjugating the rotation of angle $\theta$ to $f$: for all $x\in \RZ$, we have
\[\phi(x+\theta)=f\circ \phi(x).\]
Let $\hat\phi\colon\CZ\to \CZ$ be the homeomorphism defined by
\[\hat\phi(z) = \phi\bigl(\Re(z)\bigr)+\i \Im(z).\]
Then, $\hat\phi\colon\CZ\to \CZ$ is a $K$-quasiconformal homeomorphism with
\[K\dfn \max\bigl(\|\phi'\|_\infty,\|1/\phi'\|_\infty\bigr).\]
Now, for any $y>0$,
\[\hat\phi(x+\theta+\i y) = f\bigl(\hat\phi(x)\bigr)+iy,\]
and so, $\hat\phi$  induces a $K$-quasiconformal homeomorphism between the complex tori $\bbC/\bigl(\bbZ+(\theta+\i y)\bbZ\bigr)$ and $E(f_{\i y})$.
It follows that for $y>0$, the hyperbolic distance in $\HZ$ between $\theta+\i y$ and $\tau_f(\i y)$ is uniformly bounded and thus,
\[\lim_{\substack{y\to 0\\ y>0}}\tau_f(\i y) = \theta.\]

\section{The hyperbolic case: formation of bubbles}\label{sub-Buff-constr}

We recall the arguments of the proof of Theorem  \ref{th_mine} given in \cite{NG}. It is based on an auxiliary construction of a complex torus $\E(f)$ when $f\colon\RZ\to \RZ$ has rational rotation number and is hyperbolic. This construction will be used again in the proofs of Lemmas  \ref{lemma-bubblesize}, \ref{lemma-realendpoints},  \ref{lemma-complexendpoints}, and \ref{lemma_sizebubble}  for $f_{\omega}$ playing the role of $f$.

Let us assume $f\colon\RZ\to \RZ$ has rational rotation number $p/q$ and  has only hyperbolic periodic orbits.
The number $m\geq 1$ of attracting cycles is equal to the number of repelling cycles.
Denote by $\alpha_j$, $j\in \bbZ/(2mq)\bbZ$, the periodic points of $f$, ordered cyclically; even indices correspond to attracting periodic points and odd indices to repelling periodic points. Note that $f(\alpha_j) = \alpha_{j+2mp}$.

Let $\rho_j$ be the multiplier of $\alpha_j$ as a fixed point of $f^{\circ q}$ and $\phi_j\colon(\bbC,0)\to (\CZ,\alpha_j)$ be the linearizing map which conjugates multiplication by $\rho_j$ to $f^{\circ q}$:
\[ f^{\circ q}\circ \phi_j(z)=\phi_j(\rho_j z)\]
and is normalized by $\phi_j'(0)=1$. Then,
\[f\circ \phi_j (z) = \phi_{j+2mp}(\lambda_j\cdot z)\quad \text{with}\quad \lambda_j \dfn f'(\alpha_j).\]
In addition,  if $\eps>0$ is small enough, the linearizing map $\phi_j$ extends univalently to the strip $\Set{z \in \bbC|{|\Im(z)|<\eps}}$ and
\[\phi_j(\bbR) = (\alpha_{j-1},\alpha_{j+1}).\]

For each $j\in \bbZ/(2mq)\bbZ$, let $x_j$ be a point in $(\alpha_j,\alpha_{j+1})$, so that
\begin{itemize}
\item $f(x_j)\in (\alpha_{j+2pm}, x_{j+2pm})$ if the orbit of $\alpha_j$ attracts (i.e. $ j$  is even) and
\item $f(x_j)\in (x_{j+2pm}, \alpha_{j+2pm+1})$ if the orbit of $\alpha_j$ repels (i.e. $ j$ is odd).
\end{itemize}
This is possible since $f^{\circ q}(x_j) \in (\alpha_j, x_j)$ when $j$ is even and $f^{\circ q}(x_j) \in (x_j,a_{j+1})$ when $j$ is odd.
Similarly, let $\eps_j$ be a point on the negative imaginary axis if $j$ is even and on the positive imaginary axis if $j$ is odd, so that for all $j\in \bbZ/(2mq\bbZ)$,
\begin{itemize}
\item $|\eps_j|<\eps$, $|\lambda_j\eps_j|<\eps$ and
\item $\lambda_j  \eps_j$ is above $\eps_{j+2mp}$.
\end{itemize}

Let $C_j$ be the arc of circle with endpoints $\phi_j^{-1}(x_{j-1})$ and $\phi_j^{-1}(x_j)$ passing through $\eps_j$ and set
\[\gamma \dfn \bigcup_{j\in \bbZ/(2mq \bbZ)} \phi_j(C_j).\]
Then, $\gamma$ is a simple closed curve in $\CZ$ and $f$ is univalent in a neighborhood of $\gamma$.

\begin{figure}[htbp]
\centerline{
\begin{picture}(0,0)%
\scalebox{.315}{\includegraphics{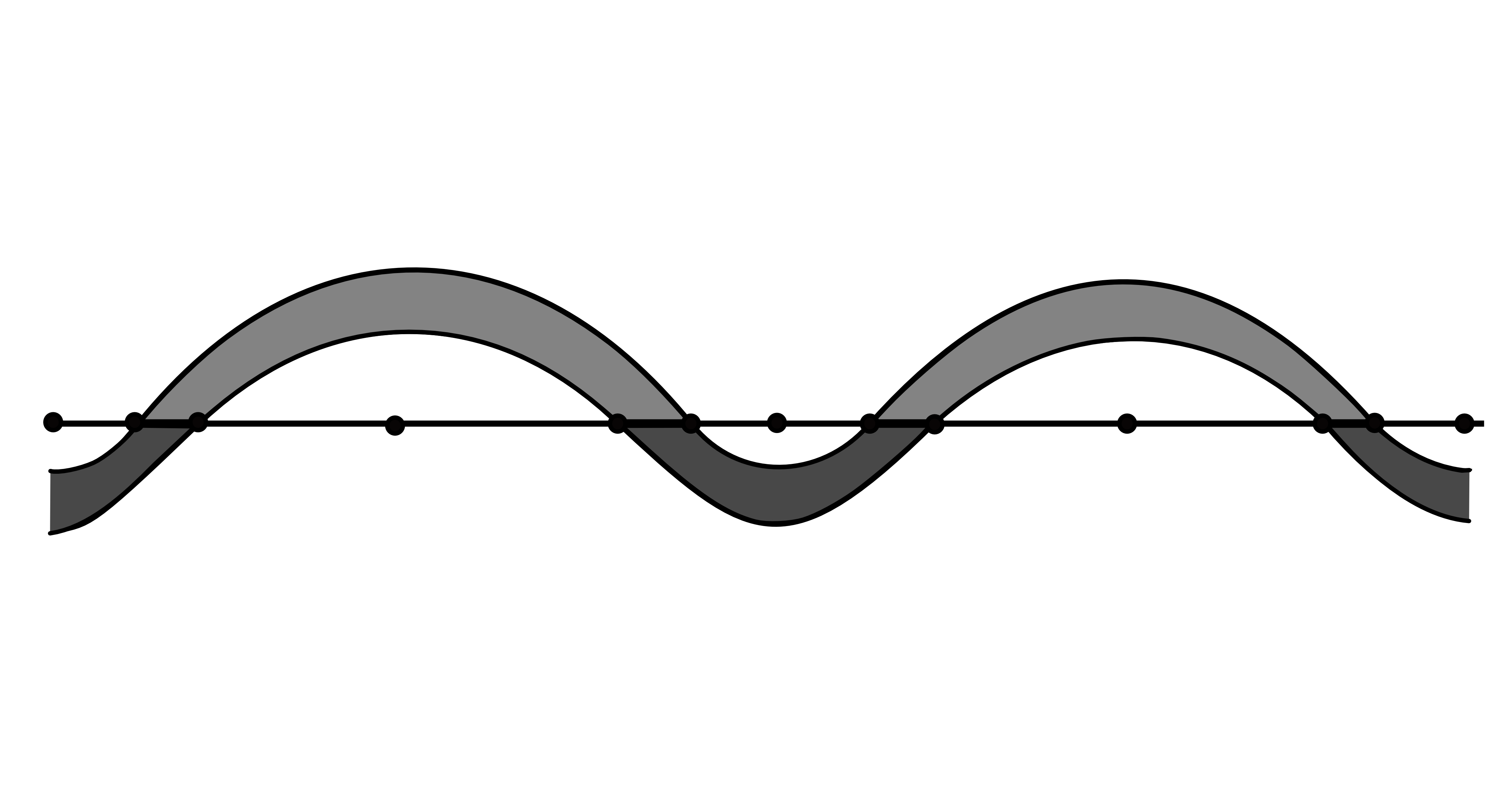}}
\end{picture}%
\setlength{\unitlength}{987sp}%
\begin{picture}(21330,10665)(1951,-12076)
\put(7551,-7136){$\alpha_1$}
\put(13451,-6036){$\alpha_0$}
\put(18300,-7136){$\alpha_1$}
\put(5950,-5850){$\gamma$}
\put(3951,-4800){$f(\gamma)$}
\put(15400,-7136){$x_0$}
\put(10800,-7136){$x_1$}
\put(21000,-7136){$x_1$}
\put(4400,-7136){$x_0$}
\end{picture}
}
\caption{
A possible choice of curve $\gamma$ for the map $f \colon \CZ\ni z \mapsto z+\frac{1}{4\pi}\sin(2\pi x)\in \CZ$ which restricts as a hyperbolic
circle diffeomorphism $f\colon \RZ\to \RZ$.
The curve $f(\gamma)$ lies above $\gamma$ in $\CZ$. The essential
annulus between $\gamma$ and $f(\gamma)$ is colored (light grey in the
upper half-plane and dark grey in the lower half-plane).
The map $f$ has an attracting fixed point at $\alpha_0 \dfn 0\in \RZ$ and
a repelling fixed point at $\alpha_1 \dfn 1/2\in \RZ$.}
\end{figure}

The attracting cycles of $f$ are above $\gamma$ in $\CZ$ and the repelling cycles are below $\gamma$ in $\CZ$.
In addition,
\[f(\gamma) = \bigcup_{j\in \bbZ/(2mq \bbZ)} \phi_{j+2mp}(\lambda_j C_j)\]
and so,  $f(\gamma)$ lies above $\gamma$ in $\CZ$.

For $\omega$ sufficiently close to $0$, the curve $f_\omega(\gamma)=f(\gamma)+\omega$ remains above $\gamma$ in $\CZ$.
The curves $\gamma$ and $f_\omega(\gamma)$ bound an essential annulus in $\CZ$. Glueing the two sides  via $f_\omega$, we obtain a complex torus $\E(f_\omega)$, which may be uniformized as $\cal{E}_\tau \dfn \bbC/(\bbZ+\tau\bbZ)$ for some appropriate $\tau\in \HZ$, the homotopy class of $\gamma$ in $\E(f_\omega)$ corresponding
to the homotopy class of $\RZ$ in $\cal{E}_\tau$. Clearly, $\E(f_\omega)$ does not depend on the choice  of $\eps_j$. We set $\bar \tau_f(\omega) \dfn \tau\in \HZ$.

According to Risler \cite[Chapter 2, Proposition 2]{Ris}, the map $\omega\mapsto \bar \tau_f(\omega)$ is holomorphic.
When $\omega\in \HZ$, the complex torus $\E(f_\omega)$ is isomorphic to $E(f_\omega)$ and the homotopy class of $\gamma$ in $\E(f_\omega)$ corresponds to the homotopy class of $\RZ$ in $E(f_\omega)$ (see \cite{NG} for details; in some sense, $E(f_{iy})$ is a limit case of $\E(f_{iy})$ as $\eps_j$ tend to zero and $\gamma$ tends to the real axis). As a consequence, $\bar \tau_f(\omega) = \tau_f(\omega)$ when $\omega\in \HZ$ is sufficiently close to $0$. This completes the proof of Theorem  \ref{th_mine} for $\omega =0$: the map $\bar \tau_f$ extends $\tau_f$ analytically to a neighborhood of zero, as required.

\begin{remark}
 Note that the curve $\E(f)$ does not depend on the choice of an analytic chart on a circle: $\bar \tau_f(0) = \bar \tau_{gfg^{-1}}(0)$ for any orientation-preserving analytic circle diffeomorphism $g$. So we can give the description of $\E(f)$ in terms of moduli of analytic conjugation, that is, in terms of the multipliers of the fixed points and the transition maps between the linearizing charts $\phi_{j}$. This description is given at the beginning of Section \ref{sec-complex}.
 \end{remark}

We will also need a following observsation:

 \begin{lemma}\label{lem_f^q}
  The modulus of $\E(f^{\circ q})$ is $q$ times  bigger then the modulus of $\E(f)$:   $\bar \tau_{f^{\circ q}}(0) = q \bar \tau_f(0)$.
 \end{lemma}

 \begin{proof}
The diffeomorphism $f$ induces an automorphism of $\E(f^{\circ q})$ of order $q$. The quotient of $\E(f^{\circ q})$ by this automorphism is isomorphic to $\E(f)$. The class of $\gamma$ in $\E(f)$ has $q$ disjoint preimages in $\E(f^{\circ q})$ which map with degree $1$ to $\gamma$. It follows that $\E(f^{\circ q})$ is isomorphic to ${\cal E}_{q\bar \tau_f(0)} \dfn \bbC/(\bbZ+q\bar \tau_f(0) \bbZ)$, the class of $\gamma$ in $\E(f^{\circ q})$ corresponding to the class of $\RZ$ in ${\cal E}_{q\bar \tau_f(0)}$.
 \end{proof}

\section{Lemma \ref{lemma-bubblesize} (size of bubbles) and Lemma \ref{lemma-realendpoints} (continuity at the real endpoints)}

We now come to our main contribution, starting with the proof of Lemma \ref{lemma_sizebubble}.
Assume $f_{\omega}\colon\RZ\to \RZ$ has rational rotation number $p/q$ and  has only hyperbolic periodic orbits. For simplicity of notation, we put $\omega=0$ and write $f=f_0 $ instead of $f_{\omega}$.
 As in Section \ref{sub-Buff-constr}, consider a simple closed curve $\gamma$ oscillating between the attracting cycles of $f$ (which are above $\gamma$ in $\CZ$) and the repelling cycles of $f$ (which are below $\gamma$ in $\CZ$), so that $f(\gamma)$ lies above $\gamma$ in $\CZ$.

The curves $\gamma$ and $f(\gamma)$ bound an essential annulus in $\CZ$. Glueing the curves via $f$, we obtain a complex torus $\E(f)$ isomorphic to ${\cal E}_\tau \dfn \bbC/(\bbZ+\tau\bbZ)$ with $\tau \dfn \bar\tau_f(0)\in \HZ$, the class of $\gamma$ in $\E(f)$ corresponding to the class of $\RZ$ in ${\cal E}_\tau$.

The projection of $\RZ$ in $\E(f)$ consists of $2m$ topological circles cutting $\E(f)$ into $2m$ annuli associated to the cycles of $f$. The moduli of the annuli depend only on multipliers of $f$.
More precisely, each attracting (respectively repelling) cycle $c$ has a basin of attraction $B_c$ for $f$ (respectively for $f^{-1}$), and the projection of $\bbH^-\cap B_c$ (respectively $\bbH^+\cap B_c$) in $\E(f)$ is an annulus $A_c$ of modulus
\[\mod A_c= \frac{\pi}{|\log \rho_c|},\]
where $\rho_c$ is the multiplier of $c$ as a cycle of $f$.

Those annuli wind around the class of $\gamma$ in $\E(f)$ with combinatorial rotation number $-p/q$.  Now, we can estimate $\E(f)$ in terms of the moduli of the annuli. It follows from a classical length-area argument (see Lemma \ref{lem-estimate-annuli} below)
that there is a representative $\tilde \tau\in \bbH$ of $\tau\in \HZ$ such that
 \[\sum_{c\text{ cycle of }f} \mod A_c\leq \frac{\Im(\tilde \tau)}{|-p+q\tilde \tau|^2}.\]
As a consequence,
 \[
\frac 12 \frac{| \tilde \tau- p/q|^2}{\Im (\tilde \tau)} \le  R_{\omega}  \dfn  \frac{1}{2\displaystyle  \pi q^2 \cdot \sum_{c\text{ cycle of }f} \mod A_c},
\]
which yields Lemma \ref{lemma_sizebubble} since
\[\sum_{c\text{ cycle of }f} \mod A_c = \sum_{c\text{ cycle of }f} \frac{\pi}{|\log \rho_c|} = \frac{1}{q}\sum_{x\in \Per(f)} \frac{\pi}{|\log \rho_x|}.\]

The proof of the first estimate in Lemma \ref{lemma_sizebubble} is completed by the following lemma.

\begin{lemma}\label{lem-estimate-annuli}

\begin{enumerate}
\item \label{estim-0}Let elliptic curve ${\cal E}_\tau =\bbC/(\bbZ+\tau\bbZ)$ contain several disjoint annuli $A_j$ which correspond to the first generator of $ H_1(E)$. Then $\Im (\tau) \ge \sum \mod A_j$.

\item  \label{estim-p/q}Let elliptic curve ${\cal E}_\tau = \bbC/(\bbZ+\tau\bbZ)$ contain several disjoint annuli $A_j$. Suppose that these annuli correspond to the element $(a,b) \sim a+b\tau$ of $ H_1(E)$, and $a$ and $b$ are coprime.  Then
\begin{equation}
 \label{eq:tau}
\frac{\Im (\tau)}{|a+b \tau|^2} \ge \sum \mod A_j.
\end{equation}
\end{enumerate}
\end{lemma}

\begin{proof}
Let us derive the second statement of this lemma from the first one. Let $k,l$ be integers satisfying $ak+bl=1$.
Apply the first statement of this lemma to the elliptic curve  $\bbC / ((a+b\tau)\bbZ+(-l+k\tau)\bbZ) $ (this is the curve ${\cal E}_\tau $ with another choice of generators). We get
\[\Im \frac{-l+k\tau}{a+b\tau}\ge \sum \mod A_j.\]
This is equivalent to \eqref{eq:tau} since
\[\Im \frac{-l+k\tau}{a+b\tau}=\Im \frac{(-l+k\tau)(a+b\bar\tau)}{|a+b \tau|^2}=\frac{(ak+bl)\Im (\tau)}{|a+b \tau|^2}= \frac{\Im (\tau)}{|a+b \tau|^2}\]

 The proof of the first statement is an application of a classical length-area argument. Namely, let $B_j = \{z \in \bbC \mid 0<\Im (z)<\mod A_j\} / \bbZ$ be the standard annulus of modulus $\mod A_j$, let $\Phi_j \colon B_j \to A_j \subset \mathcal E_{\tau}$ be biholomorphic map. Then
 $$
\int\int_{B_j} |\Phi'_j|^2 dx dy = \mathrm{Area}\, A_j;
$$
$$
 \int_0^{1} |\Phi_j'(x,iy)| dx = \mathrm{Length}\, \Phi_j([iy, iy+1]) \ge |\Phi_j(iy+1) - \Phi_j(iy)|=1,
 $$
 the latter equality holds since $\Phi_j$ is well-defined as a map of the annulus to the elliptic curve. Integrating the latter inequality along $y \in [0,\mod A_j]$, we get
 $$
\int\int_{B_j} |\Phi'_j| dx dy \ge \mod A_j.
 $$
 Now, we apply Cauchy inequality and get
 $$
2 \mod A_j \le  \int \int_{B_j} 2 |\Phi'_j| dx dy \le  \int_{B_j} |\Phi'_j|^2 +1 dx = \mod A_j + \mathrm{Area}\, A_j.
 $$
 thus $\mod A_j \le \mathrm{Area}\, A_j$. Adding these inequalities, we get
 $$
\Im (\tau) \ge  \sum  \mathrm{Area}\, A_j \ge  \sum \mod A_j.
 $$

 \end{proof}

Recall that the first statement of Lemma \ref{lemma_sizebubble} implies Lemma \ref{lemma-realendpoints}. Roughly speaking, at the real endpoint of a bubble one of the multipliers $ \rho_c$ tends to one, and the modulus of the corresponding annulus $A_c$ tends to infinity; thus our elliptic curve $\E(f_\omega)$ degenerates, and its modulus tends to the real axis.

Recall that the first statement of Lemma \ref{lemma_sizebubble} together with Lemma \ref{th_as-Denjoy} immediately implies the second statement of Lemma  \ref{lemma_sizebubble},  and thus Lemma \ref{lemma-bubblesize}.

\section{Lemma \ref{lemma-complexendpoints}: continuity at the complex endpoints of bubbles}
\label{sec-complex}

First, we explain the  main idea behind the proof for the case of zero rotation number.

We introduce another construction of the curve $\E(f)$. For each attracting fixed point $\alpha_j$,  consider  the annulus $\bbH^{-} / \{z \sim \rho_j z\}$ in the linearizing chart $\phi_j$;  for repelling fixed points,  take the annuli $\bbH^{+} / \{z \sim  \rho_j z\}$. These annuli are biholomorphic to $A_c$. Now, let us glue subsequent annuli via transition maps $\phi_{j+1}^{-1} \circ \phi_j $ between subsequent linearizing charts. The result is $\E(f)$ \footnote{ However, in this section  we shall pass to logarithmic charts $\frac {\log \phi_j}{\log \rho_j}$ for simplicity of notation.}.

For \emph{real} endpoints of bubbles, some of the moduli of $A_c$ tend to infinity, which makes $\E(f) $ to degenerate. For \emph{complex} endpoints of bubbles, the moduli of $A_c$ do not tend to infinity. We will examine the gluings $\phi_{j+1}^{-1} \circ \phi_j $  in the case when a new parabolic fixed point appears between $\alpha_j$ and $\alpha_{j+1}$ as $\omega \to \omega_0, \omega>\omega_0$; this is exactly what happens at the complex endpoint of a bubble. Roughly speaking,  we will find out that an infinite number of Dehn twists is applied to $\E(f_{\omega})$ as $\omega \to \omega_0, \omega>\omega_0$, and this makes $\E(f)$ to degenerate.
Technically, we will replace $\E(f_{\omega})$ by a quasiconformally close complex torus $E'$ (Lemma \ref{lemma-qctwist}), obtained from the same annuli via the close gluings.  Then we will prove that infinite number of Dehn twists is applied to $E'$.

At the beginning of this section, we work with individual hyperbolic map $f_{\omega}$, and for simplicity of notation we consider only the case $\omega=0$. Suppose that $\rot (f) = p/q$; then we  pass to the map $f^{\circ q}$ using  Lemma \ref{lem_f^q}.

The projection of $\RZ$ in $\E(f^{\circ q})$ cuts the torus in $2mq$ annuli $A_j$, $j\in \bbZ/(2mq) \bbZ$, which wind around the class of $\gamma$ with combinatorial rotation number $0$ and have moduli
\[\mod A_j = m_j \dfn \frac{\pi}{|\log \rho_j|}.\]
Let the strip $S_j\subset \bbC$ and the annulus $B_j\subset \CZ$ be defined by
\[S_j \dfn \set{z\in \bbC|0<\Im(z)<m_j}\quadand B_j \dfn S_j/\bbZ,\]
let $\pi \colon S_j \to B_j$ be a natural projection.
The annulus $B_j$ is biholomorphic to $A_j$. The map $z\mapsto \phi_j\circ \exp(z\cdot \log \rho_j) \colon S_j \to A_j$ induces an isomorphism $\chi_j\colon B_j\to A_j$ which extends analytically to the boundary. Consider  the points $r_j, s_j \in B_j$ given by
\begin{align*}
&r_j \dfn \pi (\tilde r_j), \quad \tilde r_j \dfn \frac{\log \phi_j^{-1}(x_j)}{\log \rho_j}\\
&s_j \dfn \pi(\tilde s_j), \quad \tilde s_j \dfn \frac{\log |\phi_j^{-1}(x_{j-1})|}{\log \rho_j}+\frac{\i \pi}{|\log \rho_j|}
\end{align*}
The point $r_j$ belongs to the lower boundary component $C_j^- \dfn \RZ$ of $B_j$, and $s_j$ belongs to its upper boundary component $C_j^+ \dfn (\bbR+\i m_j)/\bbZ$.  Note that  $\chi_j(r_j)$ is the class of $x_j$ in $\E(f^{\circ q})$, and $\chi_j (s_j) $ is the class of $x_{j-1}$ in $\E(f^{\circ q})$
 (see Figure \ref{fig:annuliintorus}).

\begin{figure}[htbp]
\centerline{
\begin{picture}(0,0)%
\scalebox{.75}{\includegraphics{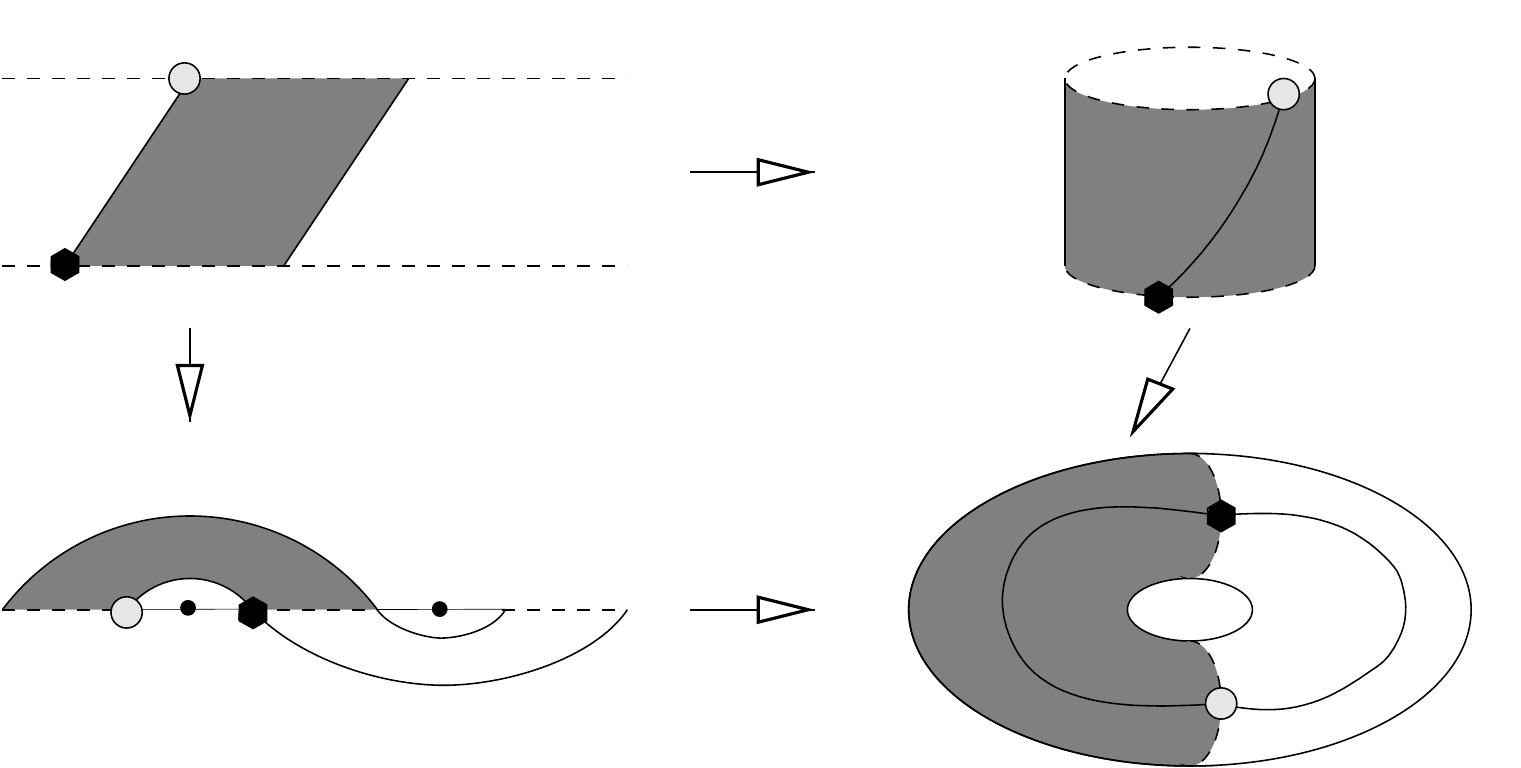}}
\end{picture}%
\setlength{\unitlength}{2960sp}%
\begin{picture}(7288,3676)(1189,-3368)
\put(782,-1000){$\scriptstyle 0$}
\put(582, -100){$\scriptstyle {\frac{i\pi }{\log \rho_j}}$}
\put(1482,-1165){$\scriptstyle \tilde r_j$}
\put(2482,-1165){$\scriptstyle 1+\tilde r_j$}
\put(1592,-2800){$\scriptstyle x_{j-1}$}
\put(3243,-2493){$\scriptstyle \alpha_{j+1}$}
\put(2022,-2800){$\scriptstyle \alpha_j$}
\put(2308,-2800){$\scriptstyle x_j$}
\put(7130,-50){$\scriptstyle s_j$}
\put(6546,-1230){$\scriptstyle r_j$}
\put(2067,106){$\scriptstyle \tilde s_j$}
\put(3067,106){$\scriptstyle 1+\tilde s_j$}
\put(2251,-1486){$\scriptstyle z\mapsto\phi_j\circ \exp(z\cdot\log \rho_j)$}
\put(3001,-3086){$\scriptstyle \gamma$}
\put(2626,-2231){$\scriptstyle f(\gamma)$}
\put(3451,-511){$\mathbf{ S_j}$}
\put(4700,-311){$\scriptstyle \pi $}
\put(4451,-2411){$\scriptstyle z\sim f(z)$}
\put(7551,-511){$\mathbf{ B_j}$}
\put(6926,-1486){$\scriptstyle \chi_j$}
\put(5683,-2659){$\mathbf A_j$}
\put(7930,-3256){$\mathbf{ \E(f)}$}
\end{picture}
}
\caption{ $A_j$, $S_j$, $B_j$ and the curve $\E(f)$. \label{fig:annuliintorus}}
\end{figure}
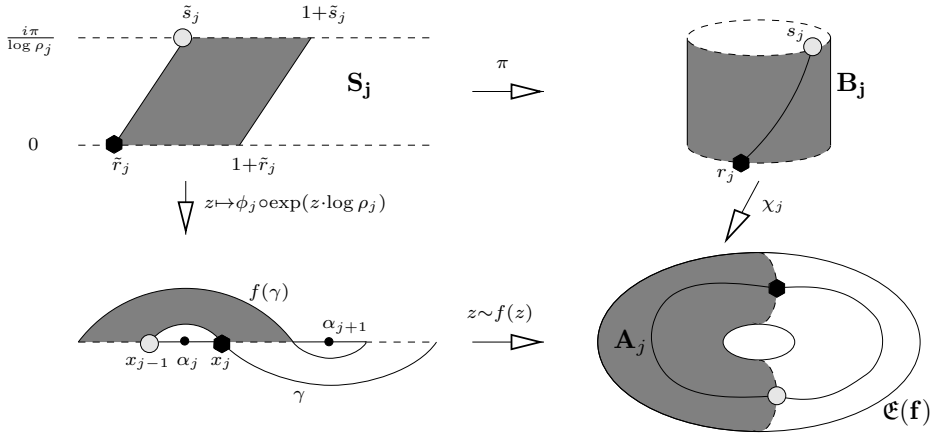

On the one hand, a complex torus $\E(f)$ is the result of glueing the lower boundary components $C_j^-$ of $B_j$ with the upper boundary components $C_{j+1}^+$ of $B_{j+1}$  via the analytic diffeomorphisms
\[
\xi_j \dfn \chi_{j+1}^{-1}\circ \chi_j\colon C_j^-\to C_{j+1}^+.
\]
Let $\delta_j$ be the projection of the segment $[\tilde r_j,\tilde s_j]$ to $\E(f)$.
Then the simple closed curve
\[
\delta \dfn \bigcup_{j\in \bbZ/(2mq)\bbZ} \delta_j.
\]
has the same homotopy class as $\gamma$ in $\E(f)$.

On the other hand, glueing the lower boundary components $C_j^-$ of $B_j$ with the upper boundary components $C_{j+1}^+$ of $B_{j+1}$  via \emph{the translations} by $z\mapsto z-r_j+s_{j+1}$, we obtain a complex torus $E'$.

It is easy to see that its modulus is
\[
\sigma \dfn \sum_{j\in \bbZ/(2mq) \bbZ}\tilde s_j-\tilde r_j.
\]
Let $\delta'_j$ be the projection of the segment $[\tilde r_j,\tilde s_j]$ to $E'$.
The homotopy class of the simple closed curve
\[\delta' \dfn \bigcup_{j\in \bbZ/(2mq)\bbZ} \delta'_j\]
in $E'$ corresponds to the homotopy class of $\sigma\bbR/\sigma\bbZ$ in ${\cal E}_\sigma$ (i.e. to its second generator).

The following lemma shows that we can replace non-trivial gluings $\xi_j$ by linear maps.

\begin{lemma}\label{lemma-qctwist}
Let $\rot (f) = p/q$. The modulus of the curve $\E(f^{\circ q})$ is $5D_f$-close to the modulus of the curve $E'$ corresponding to $f^{\circ q}$:
\[\distance_\HZ\left(q\bar \tau_{f}(0),-\frac{1}{\sigma}\right)\leq 5D_f\]
\end{lemma}

The proof of this lemma is based on the following estimate on $\xi_j$.

\begin{lemma}\label{lemma-disxi}
For any $j\in \bbZ/(2mq)\bbZ$, the distortion of the map $\xi_j$ corresponding to the map $f^{\circ q}$ is bounded by $4D_f$.
\end{lemma}

Lemma \ref{lemma-disxi} shows that $\E(f^{\circ q})$ and $E'$ are glued from the same annuli $B_j$ via the close maps, $\xi_j$ and $z \mapsto z-r_j+s_{j+1}$ respectively. The rest of the proof of Lemma \ref{lemma-qctwist} is purely technical. We construct a quasiconformal map from $\E(f^{\circ q})$ to $E'$ (actually, a tuple of maps from $B_j$ to itself) which takes $\delta_j$ to $\delta'_j$. We estimate its dilatation using Lemma \ref{lemma-disxi}. Then we refer to  Lemma \ref{lem-quasi-dist}. The detailed proof of Lemma \ref{lemma-qctwist} is in Appendix \ref{sec-lem-qctwist}.

\begin{proof}[Proof of Lemma \ref{lemma-disxi}]
The map $\xi_j\colon C_j^-\to C_{j+1}^+$ is induced by the following composition
\[\bbR\overset{e_j}\longrightarrow (0,+\infty)\overset{\phi_j}\longrightarrow (\alpha_j,\alpha_{j+1})\overset{\phi_{j+1}^{-1}}\longrightarrow (-\infty,0)\overset{e_{j+1}^{-1}}\longrightarrow  \bbR+\i m_{j+1}.\]
with
\[e_j(z)  \dfn  \exp(z\cdot \log \rho_j)\quad \text{and}\quad e_{j+1}(z) = \exp(z\cdot \log \rho_{j+1}).\]
The distortion of $e_j$ on any interval of length $1$ is $|\log \rho_j|$ which is at most $D_f$ according to Lemma \ref{th_as-Denjoy}. Similarly, the distortion of $e_{j+1}$ on any interval of length $1$ is $|\log \rho_{j+1}|\leq D_f$.

Let $x$ be any point in $(\alpha_j,\alpha_{j+1})$ and let $I\subset \RZ$ be the interval whose extremities are $x$ and $f^{\circ q}(x)$.
To complete the proof, it is enough to show that
\[\distortion_I (\phi_{j}^{-1})\leq D_f \quad\text{and}\quad \distortion_I (\phi_{j+1}^{-1})\leq D_f.\]
We will only prove this result for $\phi_j$ in the case where the orbit of $\alpha_j$ is attracting. The other cases are dealt similarly and left to the reader.

On $I$, the linearizing map $\phi_j$ is the limit of the maps $\varphi_n \dfn (f^{\circ nq}-\alpha_j)/\rho_j^n$. Since $I$ is disjoint from all its iterates, Denjoy's Lemma \ref{denjoy1} yields
\[\distortion_I \varphi_n = \distortion_I f^{\circ qn} \leq D_f.\]
Passing to the limit as $n$ tends to $\infty$ shows that $\distortion_I \phi_j\leq D_f$  as required.
\end{proof}

Now, we come to the proof of Lemma \ref{lemma-complexendpoints}.
\begin{proof}
Without loss of generality, we suppose that  $\omega_0=0$.
According to Lemma \ref{lemma_sizebubble}, we know that for $\omega>0$ close to $0$, $\bar\tau_f(\omega)$ remains in a subdisk of $\HZ$ tangent to the real axis at $p/q$. Lemma \ref{lem_f^q} shows that $q \bar\tau_{f_{\omega}}(0) = \bar\tau_{(f_{\omega})^{\circ q}}(0)$. So, it is enough to prove that $\bar\tau_{(f_{\omega})^{\circ q}}(0)$ tends to $0$ tangentially to the segment $[0,\eps)\in \RZ$ and is located in between two horocycles  at  $0$.
According to Lemma \ref{lemma-qctwist}, the hyperbolic distance in $\HZ$ between $ \bar\tau_{(f_{\omega})^{\circ q}}(0)$ and $-1/\sigma_{\omega}$ (where $\sigma_\omega$ corresponds to the map $f_{\omega}^{\circ q}$ )  is uniformly bounded as $\omega>0$ tends to $0$. So, it is enough to show that the imaginary part of $\sigma_\omega$ is bounded and that the real part of $\sigma_\omega$ tends to $-\infty$.

We modify the notation of Section \ref{sub-Buff-constr}. Now, we have a family $f_{\omega}^{\circ q}$ of hyperbolic diffeomorphisms with $\rot (f_{\omega}^{\circ q})=0$, $\omega \in (0, \eps]$. For $\omega =0$, the map $f_0=f$ is not hyperbolic.

As in Section \ref{sub-Buff-constr}, let $\alpha_j(\omega)$, $j\in \bbZ/(2mq\bbZ)$, be all fixed points of $f_{\omega}^{\circ q}$  with multipliers $\rho_{\omega,j}$ and with linearizing charts  $\phi_{\omega,j}$. For the correct numbering, $\alpha_j(\omega)$ depend holomorphically on $\omega$ and $\alpha_j := \lim_{\omega \to 0} \alpha_j(\omega)$ are all \emph{hyperbolic} fixed  points of $f^{\circ q}$.  Then $\rho_j := \lim _{\omega \to 0}\rho_{\omega,j}$ are their multipliers, and $\phi_j:= \lim_{\omega \to 0} \phi_{\omega,j}$ are their linearizing charts; the latter convergence is guaranteed only in neighborhoods of hyperbolic cycles of $f$.

Now, we want to introduce points $x_j$ not depending on $\omega$ (this is a main advantage of the modified notation).

For each $j\in \bbZ/(2mq)\bbZ$, let $x_j$ be a point in $(\alpha_j,\alpha_{j+1})$, so that
\begin{itemize}
\item $f^{\circ q}(x_j)\in (\alpha_{j}, x_{j})$ if $\alpha_j$ attracts (i.e. $ j$  is even) and
\item $f^{\circ q}(x_j)\in (x_{j}, \alpha_{j+1})$ if $\alpha_j$ repels (i.e. $ j$ is odd).
\end{itemize}
These are exactly the  conditions from Section \ref{sub-Buff-constr} for the map $f^{\circ q}$ with $\rot (f^{\circ q})=0$, but here $f^{\circ q}$ is not hyperbolic.
Note that since the parabolic fixed points disappear as $\omega>0$ increases, the graph of $f^{\circ q}$ lies above the diagonal near those points. As a consequence, each parabolic fixed point of $f^{\circ q}$ lies in an interval of the form $(\alpha_j,\alpha_{j+1})$ with $\alpha_j$ repelling and $\alpha_{j+1}$ attracting.

 Finally, set
\[\tilde r_{\omega,j} \dfn \frac{\log \phi_{\omega,j}^{-1}(x_j)}{\log \rho_{\omega,j}},\quad  \tilde s_{\omega,j} \dfn \frac{\log |\phi_{\omega,j}^{-1}(x_{j-1})|}{\log \rho_{\omega,j}}+\frac{\i \pi}{|\log \rho_{\omega,j}|}\]
and
\[ \sigma_\omega \dfn  \sum_{j\in \bbZ/(2mq)\bbZ} \tilde s_{\omega,j}-\tilde r_{\omega,j}.\]
This definition agrees with the notation of Lemma \ref{lemma-qctwist}. $\sigma_{\omega}$  is equal to the number $\sigma$ from Lemma \ref{lemma-qctwist} corresponding to $f_{\omega}^{\circ q}$.

Now, it suffices to prove that the imaginary part of $\sigma_\omega$ is bounded and that the real part of $\sigma_\omega$ tends to $-\infty$.
Thus its modulus tends to $0$ in between two horocycles.

The imaginary part of $\sigma_\omega$ is equal to the sum of $\mod A_j$,
\[\Im(\tilde r_{\omega,j})=0\quad\text{and}\quad \Im(\tilde s_{\omega,j})=\frac{\pi}{|\log \rho_{\omega,j}|}\underset{\omega>0,\omega\to 0} \longrightarrow \frac{\pi}{|\log \rho_{j}|},\]
and we see that it remains bounded as $\omega>0$ tends to $0$.

If $f^{\circ q}$ has no parabolic fixed point on the interval $(\alpha_j,\alpha_{j+1})$, then $\phi_{\omega,j}^{-1}\to \phi_j^{-1}$ on the interval $(\alpha_j,\alpha_{j+1})$. It follows that $\Re(\tilde r_{\omega,j})$ and $\Re(\tilde s_{\omega,j+1})$ remain bounded.
If $f$ has a parabolic periodic point on  the interval $(\alpha_j,\alpha_{j+1})$, then $\alpha_j$ is repelling and $\alpha_{j+1}$ is attracting. Either the two quantities $\log \phi_{\omega,j}^{-1}(x_j)$ and $\log |\phi_{\omega,j+1}^{-1}(x_j)|$ tend to $+\infty$, or one remains bounded and the other tends to $+\infty$. Since $\log \rho_{\omega,j}\to \log \rho_j>0$ and $\log \rho_{\omega,j+1}\to \log \rho_{j+1}<0$, in both cases we have
\[\Re(\tilde s_{\omega,j+1}-\tilde r_{\omega,j})\underset{\omega>0,\omega\to 0}\longrightarrow -\infty.\qedhere\]
This finishes the proof.

\end{proof}

\section{Continuity of the boundary function \texorpdfstring{$\bar \tau_f$}{(bar)τ\_f}}

We now prove Lemma \ref{lem-continuity}. It is enough to prove that $\bar \tau_f$ is continuous at $\omega=0$.

\subsection{Irrational rotation number}

If  $\rot(f)$ is irrational, then $\bar \tau_f (0) =  \rot(f)$ due to the definition of $\bar \tau_f$.

Let $I\subset \RZ$ be a small neighborhood of $0$ such that for $\omega\in I$, the periods of the periodic orbits of $f_\omega$
are at least $N$. For  $\omega \in I$, either $\bar \tau_f(\omega)=\rot(f_\omega)$, or according to Lemma \ref{lemma_sizebubble},
\[\bigl|\bar \tau_f(\omega)-\rot(f_\omega)\bigr|\leq \frac{D_f}{2 \pi N^2}.\]
Thus, $\bar \tau_f(I)$ is located within $D_f/(2\pi N^2)$-neighborhood of $\Set{\rot(f_\omega)| \omega \in I}$. The result follows since $\omega\mapsto \rot(f_\omega)$ is continuous.

\subsection{Rational rotation number}
\label{subsub:contin-at-rational}

It is sufficient to prove that
\[ \lim _{\omega >0, \omega \to 0 } \bar \tau_f(\omega) = \frac pq = \bar \tau_f (0).\]
Indeed, if we apply this result to the  diffeomorphism $x\mapsto -f(-x)$ we get
\[\lim _{\omega<0, \omega\to 0 } \bar \tau_f(\omega) = \frac pq = \bar \tau_f (0).\]
(see Remark \ref{remark-leftRight} for details). There are the following cases.

\begin{enumerate}

 \item  $f$ is hyperbolic. The continuity of  $\bar \tau_f$ at $0$ follows directly from Theorem \ref{th_mine}.

 \item  $f$ has at least one parabolic cycle.

 \begin{itemize}
 \item $0$ is not a left endpoint of a bubble: all $q$-periodic orbits of $f$ disappear as $\omega$ increases ($\rot(f_{\omega})>p/q$ for $\omega>0$). In this case, the proof is literally the same as in the case of irrational rotation number.

\item $0$ is a real left endpoint of a bubble.  The result follows from Lemma \ref{lemma-realendpoints}.

\item $0$ is a complex left endpoint of a bubble. The result follows from Lemma \ref{lemma-complexendpoints}.
\end{itemize}
\end{enumerate}

\section{Proof of the Main Theorem}
\label{sec_contin}

The map
\[\CZ\ni z \mapsto \exp(2\pi\i z)\in \bbC-\set{0}\]
is an isomorphism of Riemann surfaces.
It conjugates $\tau_f \colon \HZ \to \HZ $ to a holomorphic function $g\colon\bbD-\set{0}\to \bbD-\set{0}$ and
$\bar\tau_f \colon\RZ\to \HZbar$ to a continuous function $h\colon\partial \bbD\to \overline \bbD$.
Since $g$ is bounded,  it extends holomorphically at $0$. According to the previous study,
\[\text{for almost every }t\in\RZ,\quad \lim_{r \to 1, r<1} g(r e^{2\pi \i t}) = h(e^{2\pi \i t}).\]
The Main Theorem is therefore a consequence of the following classical result.

\begin{lemma}
Let $g \colon \bbD \to \bbC$ be a bounded holomorphic function and $h\colon\partial \bbD\to \bbC$ be a continuous function such that:
\[\text{for almost every }t\in\RZ,\quad \lim_{r \to 1, r<1} g(r e^{2\pi \i t}) = h(e^{2\pi \i t}).\]
Then, $h$ extends $g$ continuously to $\overline\bbD$.
\end{lemma}

\begin{proof}
The real and imaginary parts of $g$ are harmonic functions.
Due to the Poisson formula (applied to both $\Re g$ and $\Im g$) for $|z|<r$ we have
\begin{equation}
\label{eq:Pois}
	g(z) = \frac{1}{2\pi} \int_{0}^{2\pi } g(r e^{i\alpha})  P(re^{i\alpha},z)  \dx \alpha ,
\end{equation}
 where $P$ is the Poisson kernel,
\[
P(re^{i\alpha}, R e^{i \beta}) = \frac{r^2-R^2}{r^2+R^2- 2rR \cos  (\alpha-\beta)}.
\]
The integrand in \eqref{eq:Pois} is bounded as $r$ tends to $1$, and it tends to $h(e^{i\alpha}) P(e^{i\alpha},z) $ almost everywhere. Due to the Lebesgue
bounded convergence theorem,
\[
g(z) = \frac{1}{2\pi} \int_{0}^{2\pi } h(e^{i\alpha})  P(e^{i\alpha},z) \dx \alpha.
\]
Due to the Poisson theorem, the right-hand side provides the solution of the Dirichlet boundary problem for Laplace equation.
Thus $\Re g$ and $\Im g$ satisfy
\[\lim_{z \to e^{i\alpha}} \Re g(z)= \Re h(e^{i\alpha}),\quad \lim_{z \to e^{i\alpha}} \Im g(z)= \Im h(e^{i\alpha}).\qedhere\]
\end{proof}

 \appendix

\section{Behavior of \texorpdfstring{$\tau_f$}{τ\_f} near \texorpdfstring{$+\i\infty$}{+i∞}}\label{sec_inf}

Here, we study the behavior of $\tau_f(\omega)$ as the imaginary part of $\omega$ tends to $+\infty$.
The map $\CZ\ni z \mapsto \exp(2\pi\i z)\in \bbC-\set{0}$
is an isomorphism of Riemann surfaces. Thus, $\CZ$ may be compactified as a Riemann surface $\CZbar$ isomorphic to the Riemann sphere, by adding two points $+\i\infty$ and $-\i\infty$ (the notation suggests that $\pm \i \infty$ is the limit of points $z\in \CZ$ whose imaginary part tends to $\pm \infty$).  We shall denote by
\[\HpmZbar=(\HpmZ)\cup (\RZ)\cup \set{\pm\i\infty}\]
the closure of $\HpmZ$ in $\CZbar$.

The following construction is usually referred to as {\em conformal welding}. It is customarily studied in the case of non-smooth circle homeomorphisms and is trivial in the
case of analytic circle diffeormorphisms.

The analytic  circle diffeomorphism $f$ may be viewed as an analytic diffeomorphism between the boundary of $\HpZbar$ and the boundary of $\HmZbar$.
If we glue $\HpZbar$ to $\HmZbar$ via $f$, we obtain a Riemann surface which is isomorphic to $\CZbar$. We may choose the isomorphism $\phi$ such that
$\phi(\pm\i\infty)=\pm\i \infty$. Such an isomorphism is not unique, but it is unique up to addition of a constant in $\CZ$. It restricts to univalent maps $\phi^\pm\colon\HpmZ\to
\CZ$ which extend univalently to neighborhoods of $\HpmZbar$ and satisfy $\phi^-\circ f = \phi^+$ near the boundary of $\HpZbar$.

Holomorphy of $\phi^\pm$ near $\pm\i\infty$ yields that
\[\phi^\pm(z) = z + C^\pm + o(1)\text{ as }z\to \pm\i\infty\]
for appropriate constants $C^\pm\in \CZ$.
Since $\phi$ is unique up to addition of a constant, the difference
\[C_f\dfn C^+-C^-\]
only depends on $f$ and will be referred as the {\em welding constant} of $f$.

\begin{theorem}(Behavior near $+i\infty$)
Let $f\colon\RZ\to \RZ$ be an orientation preserving analytic circle diffeomorphism and let $C_f$ be its welding constant. As
$\omega$ tends to  $+\i\infty$ in $\CZ$,
\[
\tau_f(\omega) = \omega +  C_f + o(1).
\]
\end{theorem}

The proof goes as follows.

\setcounter{step}{0}
\step
Recall that $A_\omega\subset \CZ$ is the annulus bounded by the circles $\RZ$ and $\RZ+\omega$.
The isomorphism between the complex torus $E(f_\omega)$ and $\cal{E}_{\tau_f(\omega)}$ induces a univalent map $\phi_\omega\colon A_\omega\to \CZ$ which extends univalently to a
neighborhood of the closed annulus $\overline A_\omega$, with $\phi_\omega(f_\omega(z)) = \phi_\omega(z) + \tau_f(\omega)$ in a neighborhood of $\RZ$.

\step As $\omega\to +\i \infty$, the sequence of univalent maps
\[\phi_\omega^+\colon z\mapsto \phi_\omega(z)-\phi_\omega(0)\]
converges locally uniformly in $\HpZ$ to a limit $\phi^+\colon\HpZ\to \CZ$, and the sequence of univalent maps
\[\phi_\omega^-\colon z\mapsto \phi_\omega(z+\omega)-\phi_\omega\bigl(f(0)+\omega\bigr)\]
converges  locally uniformly in $\HmZ$ to a limit $\phi^-\colon\HmZ\to \CZ$. In addition, the maps
$\phi^\pm\colon\HpZ\to \CZ$  form a pair of univalent maps provided by the welding construction.

\step Comparing constant Fourier coefficients of $\phi_\omega$, $\phi^+$ and $\phi^-$, we deduce that
as $\omega\to +\i\infty$, we have
\[C^++\phi_\omega(0) = -\omega+ C^-+\phi_\omega\bigl(f(0)+\omega\bigr) + o(1), \]
whence
\[\tau_f(\omega) = \phi_\omega\bigl(f(0)+\omega\bigr)-\phi_\omega(0) = \omega + C^+-C^- + o(1)=\omega+C_f+o(1).\]

\subsection{The map \texorpdfstring{$\phi_\omega$}{φ\_ω}}

Let $\delta>0$ be sufficiently tiny so that $f\colon\RZ\to \RZ$ extends univalently to the annulus $B_\delta\dfn \Set{z\in \CZ|{\delta>|\Im(z)|}}$. Set
\[A_\omega^+ \dfn  A_\omega\cup B_\delta\cup \bigl(\omega+ f(B_\delta)\bigr).\]
The complex torus $E(f_\omega)$ is the quotient of $A_\omega^+$ where $z\in B_\delta$ is identified to $f_\omega(z)\in f(B_\delta)+\omega$.

An isomorphism between $E(f_\omega)$ and $\cal{E}_\tau\dfn \bbC/(\bbZ+\tau\bbZ)$ sending the homotopy class of $\RZ$ in $E(f_\omega)$ to the homotopy class of $\RZ$ in
$\cal{E}_{\tau_f(\omega)}$ will lift to a univalent map $\phi_\omega\colon A_\omega^+\to \CZ$ sending $\RZ$ to a curve homotopic to $\RZ$, preserving orientation. The following
relation then holds on $B_\delta$:
\[\phi_\omega \circ f_\omega = \phi_\omega + \tau_f(\omega).\]

\subsection{Convergence of \texorpdfstring{$\phi_\omega^\pm$}{φ\_ω\textasciicircum±}}

As $\omega\to +\i\infty$, the open sets $A_\omega^+$ eat every compact subset of $\HpZ\cup B_\delta$. The sequence of univalent maps $\phi_\omega^+\colon A_\omega^+\to \CZ$ defined
by
\[\phi_\omega^+(z) \dfn \phi_\omega(z)-\phi_\omega(0)\]
is normal and any limit value $\phi^+\colon\HpZ\cup B_\delta$ satisfies $\phi^+(0)=0$. It cannot be constant since each $\phi_\omega^+$ sends $\RZ$ to a homotopically nontrivial
curve in $\CZ$ passing through $0$. So, any limit value $\phi^+\colon\HpZ\cup B_\delta\to \CZ$ is univalent.

Similarly, as $\omega\to +\i\infty$, the open sets
\[A_\omega^-\dfn -\omega + A_\omega^+\]
eat every compact subset of $\HmZ\cup f(B_\delta)$. In addition, the sequence of univalent maps
$\phi_\omega^-\colon A_\omega^-\to \CZ$ defined by
\[\phi_\omega^-(z)\dfn \phi_\omega(z+\omega)-\phi_\omega\bigl(f(0)+\omega\bigr)\]
is normal and any limit value $\phi^-\colon\HZ\cup f(B_\delta)\to \CZ$ is univalent and satisfies $\phi^-\bigl(f(0)\bigr)=0$.

Passing to the limit on the following relation, valid on $B_\delta$:
\begin{align*}
\phi_\omega^-\circ f (z) &=\phi_\omega\bigl(f(z)+\omega) - \phi_\omega\bigl(f(0)+\omega\bigr) \\
& = \phi_\omega(z) + \tau_f(\omega) - \phi_\omega\bigl(f(0)+\omega\bigr) = \phi_\omega(z) -\phi_\omega(0) = \phi_\omega^+(z),
\end{align*}
we get the following relation, valid on $B_\delta$:
\[\phi^-\circ f = \phi^+.\]

It follows that the pair $(\phi^-,\phi^+)$ induces an isomorphism from $\big(A_\omega^+\sqcup A_\omega^-\bigr)/f$ (we identify $z\in B_\delta\subseteq A_\omega^+$ to $f(z)\in
f(B_\delta)\subseteq A_\omega^-$) to $\CZ$. Therefore, $\phi^-$ and $\phi^+$ coincide with the unique isomorphisms arising from the welding construction, normalized by the
conditions $\phi^+(0)=\phi^-\bigl(f(0)\bigr) = 0$.
This uniqueness shows that there is only one possible pair of limit values. Thus, the sequences
$\phi_\omega^-\colon A_\omega^-\to \CZ$ and $\phi_\omega^+\colon A_\omega^+\to \CZ$ are convergent.

\subsection{Comparing Fourier coefficients}

Note that $z\mapsto \phi_\omega^\pm(z)-z$ and $z\mapsto \phi^\pm(z)$ are well-defined on $\RZ$ with values in $\bbC$. The previous convergence implies:
\[C_\omega^+\dfn \int_{\RZ} \bigl(\phi_\omega^+(z)-z\bigr)\, {\rm d}z\underset{\omega\to +\i\infty}\longrightarrow C^+\dfn \int_{\RZ} \bigl(\phi^+(z)-z\bigr)\, {\rm d}z\]
and
\[C_\omega^-\dfn \int_{\RZ} \bigl( \phi_\omega^- (z)-z \bigr)\, {\rm d}z\underset{\omega\to +\i\infty}\longrightarrow C^-\dfn \int_{\RZ} \bigl(\phi^-(z)-z\bigr)\, {\rm
d}z.\]
Since $\phi_\omega$ is holomorphic on $A_\omega^+$, we have
\[\int_{\RZ} \bigl(\phi_\omega(z)-z\bigr)\, {\rm d}z = \int_{\omega+ \RZ}\bigl( \phi_\omega(z)-z\bigr)\, {\rm d}z =
 \int_{\RZ}\bigl( \phi_\omega(t+\omega)-t\bigr)\, {\rm d}t-\omega.\]
 Thus,
 \begin{align*}
 C_\omega^+ & \dfn \int_{\RZ} \bigl(\phi_\omega^+(z)-z\bigr)\, {\rm d}z \\
 & =  \int_{\RZ} \bigl(\phi_\omega(z)-z\bigr)\, {\rm d}z - \phi_\omega(0)\\
 & =  \int_{\RZ}\bigl( \phi_\omega(t+\omega)-t\bigr)\, {\rm d}t-\omega - \phi_\omega(0) \\
 &=  \int_{\RZ} \bigl(\phi_\omega^-(t)-t\bigr)\, {\rm d}t -\omega + \phi_\omega\bigl(f(0)+\omega\bigr)-\phi_\omega(0)= C_\omega^--\omega+\tau_f(\omega).
 \end{align*}
 As $\omega\to +\i\infty$, we therefore have
 \[C^++o(1) = C^-+o(1)-\omega+\tau_f(\omega)\]
 which yields
 \[\tau_f(\omega) = \omega + C^+-C^- +o(1) = \omega + C_f+o(1).\]

\section{Tsujii's theorem}
\label{sec-tsujii}

For completeness, we now present a proof of Tsujii's Theorem \ref{th_Tsujii} which we believe is a simplification of the original one, although the ideas are essentially the same.
The main argument in Tsujii's proof is the following.

\begin{proposition}\label{prop_Tsujii}
Let $f\colon\RZ\to \RZ$ be a ${\cal C}^2$-smooth orientation preserving circle diffeomorphism with irrational rotation number $\theta\in\RZ$. If $p/q$ is an approximant to $\theta$ given by the continued fraction algorithm, then there is an $\omega\in \RZ$ satisfying
\[ |\omega|<e^{D_f}\cdot |\theta-p/q|\quadand \rot(f_\omega)=p/q.\]
\end{proposition}

\begin{proof}
According to a Theorem of Denjoy, there is a homeomorphism $\phi\colon\RZ\to \RZ$ such that
$\phi(x+\theta)=f\circ \phi(x)$ for all $x\in \RZ$.

Without loss of generality, let us assume that $\theta<p/q$ and set $\delta\dfn p-q\theta$. Let $T\subset \RZ$ be the union of intervals
\[T\dfn \bigcup_{1\leq j\leq q} T_j\quad\text{with}\quad T_j\dfn (j\theta,j\theta+\delta).\]
Since $p/q$ is an approximant of $\theta$, this is a disjoint union of $q$ intervals of length $\delta$.
According to Lemma \ref{lemma_tonelli} below, we may choose $t\in \RZ$ such that the Lebesgue measure of $\phi(T+t)$ is at most $q\delta$.

Now, set $x\dfn \phi(t)$ and for $j\in \bbZ$, set
\[x_j\dfn f^{\circ j}(x) = \phi(t+j\theta)\quadand I_j\dfn (x_j,x_{j-q}) = \phi(T_j).\]
The intervals $I_1$, $I_2=f(I_1)$, \ldots, $I_q=f^{\circ q}(I_1)$ are disjoint and the sum of their lengths satisfies
\[\sum_{j=1}^q |I_j|\leq q\delta = q^2\cdot |\theta-p/q|.\]
As $\omega\in\RZ$ increases from $0$, the rotation number $\rot(f_\omega)\in\RZ$ increases from $\theta$, and there is a first $\omega_0$ such that $\rot(f_{\omega_0})=p/q$.
For $j\in [0,q]$, set
\[y_j\dfn (f_{\omega_0})^{\circ j}(x)\quadand z_j\dfn f^{\circ (q-j)}(y_j).\]
Finally, for $j\in [1,q]$, set
\[J_j\dfn  \bigl(f(y_{j-1}),y_j\bigr) = \bigl(f(y_{j-1}),f(y_{j-1})+\omega_0\bigr)\quadand
K_j\dfn (z_{j-1},z_j).\]
Then, $(z_0,z_1,\ldots,z_q)$ is a subdivision of $(z_0,z_q)$ (see Figure \ref{fig:intervals}).

 \begin{figure}[htbp]
 \centerline{
 \begin{picture}(0,0)%
 \includegraphics{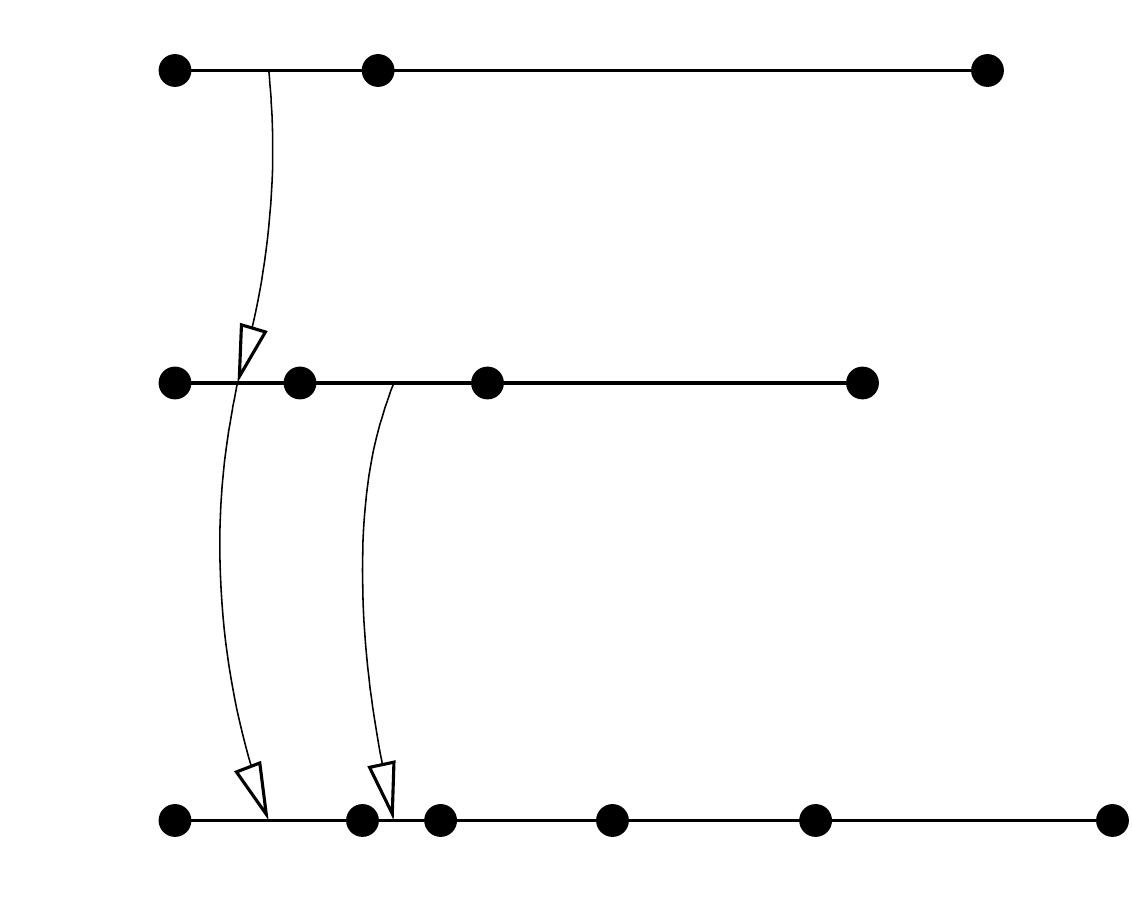}%
 \end{picture}
 \setlength{\unitlength}{3947sp}%
 \begin{picture}(5423,4311)(961,-5839)
 \put(5551,-1711){$\scriptstyle x_{1-q}$}%
 \put(2776,-4486){$\scriptstyle f^{\circ (q-2)}$}%
 \put(2326,-2611){$\scriptstyle f$}%
 \put(1426,-4486){$\scriptstyle f^{\circ (q-2)}$}%
 \put(2276,-3211){$\scriptstyle f(y_1)$}%
 \put(2776,-3211){$\scriptstyle J_2$}%
 \put(2176,-1711){$\scriptstyle J_1$}%
 \put(1576,-1711){$\scriptstyle x_1$}%
 \put(976, -1936){$\scriptstyle I_1$}%
 \put(1051,-3436){$\scriptstyle I_2$}%
 \put(1126,-5536){$\scriptstyle I_q$}%
 \put(1426,-5311){$\scriptstyle x_q=z_0$}%
 \put(2801,-5761){$\scriptstyle K_2$}%
 \put(3076,-5311){$\scriptstyle z_2$}%
 \put(2476,-5311){$\scriptstyle z_1$}%
 \put(2226,-5761){$\scriptstyle K_1$}%
 \put(3701,-5311){$\scriptstyle z_{q-1}$}%
 \put(4676,-5311){$\scriptstyle y_q=z_q$}%
 \put(4126,-5761){$\scriptstyle J_q=K_q$}%
 \put(1576,-3211){$\scriptstyle x_2$}%
 \put(6151,-5311){$\scriptstyle x_0$}%
 \put(4951,-3211){$\scriptstyle x_{2-q}$}%
 \put(2776,-1711){$\scriptstyle y_1$}%
 \put(3376,-3211){$\scriptstyle y_2$}%
 \end{picture}
 }
 \caption{The intervals $I_j$, $J_j$ and $K_j$. \label{fig:intervals}}
 \end{figure}
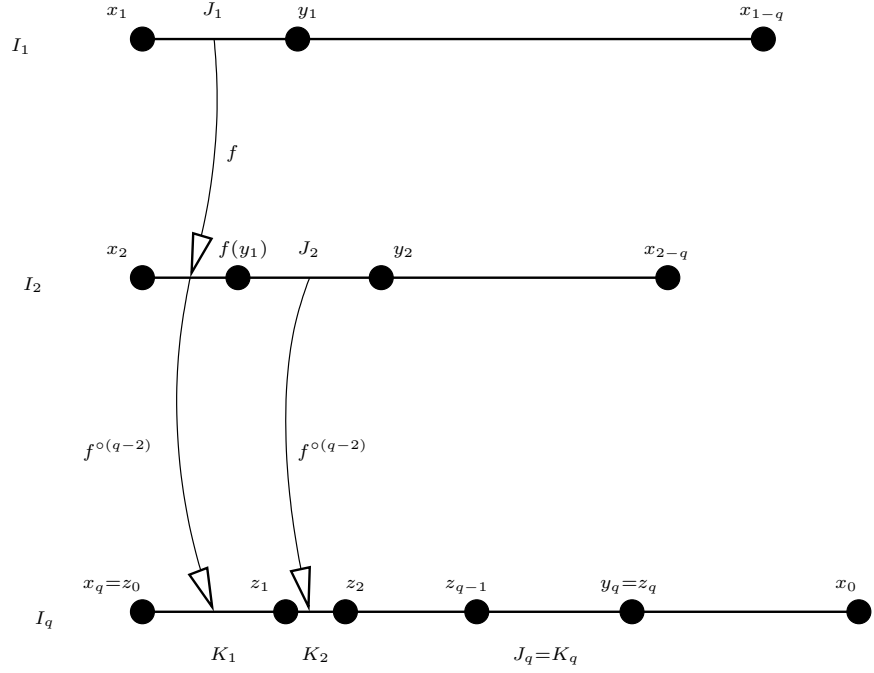

As $\omega$ increases from $0$ to $\omega_0$, the point $(f_\omega)^{\circ q}(x)$ increases from $x_q$ to $y_q$ but remains in $I_q$ since $\rot(f_\omega)$ remains less than
$p/q$. Thus,
$(z_0,z_q)=(x_q,y_q)\subseteq I_q$ and so,
\[|I_q| \geq |z_q-z_0| =\sum_{j=1}^q |K_j|.\]
In addition, $J_j\subset I_j$ and $K_j = f^{\circ (q-j)}(J_j)$. It follows from Denjoy's Lemma \ref{denjoy1} that
\[\frac{|K_j|}{|I_q|}\geq e^{-D_f} \frac{|J_j|}{|I_j|}= e^{-D_f}\frac{\omega_0}{|I_j|}.\]
Now, according to the Cauchy-Schwarz Inequality, we have
\[q^2 = \left(\sum_{j=1}^q\sqrt{|I_j|} \cdot \frac{1}{\sqrt{|I_j|}}\right)^2\leq \left(\sum_{j=1}^q |I_j|\right)\cdot \left(\sum_{j=1}^q \frac{1}{|I_j|}\right) \leq q^2\cdot
|\theta-p/q|\cdot \sum_{j=1}^q \frac{1}{|I_j|}.\]
Thus,
\[|I_q|\geq \sum_{j=1}^q |K_j| \geq e^{-D_f}{\omega_0 |I_q|}\cdot \sum_{j=1}^q  \frac{1}{|I_j|}
\geq \frac{e^{-D_f}\omega_0 |I_q|}{ |\theta-p/q|}\]
and so,
\[\omega_0\leq  e^{D_f}\cdot  |\theta-p/q|.\qedhere\]
\end{proof}

\begin{lemma}\label{lemma_tonelli}
Let $\phi\colon\RZ\to \RZ$ be a homeomorphism. Then, for any measurable set $T\subseteq \RZ$, there is a $t\in\RZ$ such that
\[{\rm Leb}\bigl(\phi(T+t)\bigr)\leq {\rm Leb}(T).\]
\end{lemma}

\begin{proof}
Let $\mu$ be the Lebesgue measure on $\RZ$.
According to Tonelli's theorem,
\begin{align*}
\int_{t\in \RZ} \mu\bigl(\phi(T+t)\bigr)\dx t & = \int_{t\in\RZ}\left(\int_{u\in T+t} \dx(\phi^* \mu)\right)\dx\mu\\
&= \int_{u\in \RZ}\left(\int_{t\in -T+u}\dx \mu\right)\dx(\phi^*\mu)\\
&= \int_{u\in \RZ}\mu(T)\dx(\phi^*\mu)\\
&= \mu(T)\cdot \mu\bigl(\phi(\RZ)\bigr) =\mu(T).
\end{align*}
So, the average of $\mu\bigl(\phi(T+t)\bigr)$ with respect to $t$ is equal to $\mu(T)$ and the result follows.
\end{proof}

Theorem \ref{th_Tsujii} follows easily from Proposition \ref{prop_Tsujii}: for $\beta>0$, let $S_\beta$ be the set of $\omega\in \RZ$ such that $\rot(f_\omega)$ is irrational
and such that there are infinitely many $p,q\in \bbZ$ satisfying $\bigl|\rot(f_\omega)-p/q\bigr|<1/q^{2+\beta}$. The set of
$\omega\in \RZ$ such that $\rot(f_\omega)$ is Liouville is the intersection of the sets $S_\beta$. So, it is sufficient to show that the ${\rm Leb}(S_\beta)=0$ for all
$\beta>0$. Note that
\[S_\beta = \limsup_{q\to +\infty} S_{\beta,q}\]
where $S_{\beta,q}$ is the set of $\omega\in \RZ$ such that $\rot(f_\omega)$ is irrational and such that $\bigl|\rot(f_\omega)-p/q\bigr|<1/q^{2+\beta}$ for some approximant
$p/q$ of $\rot(f_\omega)$.

Proposition \ref{prop_Tsujii} implies that $S_{\beta,q}$ is located in the $C/q^{2+\beta}$-neighborhood of the union of $q$ intervals where the rotation number is rational
with denominator $q$, where $C \dfn e^{D_f}$. So,
\[{\rm Leb}(S_{\beta,q})\leq 2q\cdot \frac{C}{q^{2+\beta}} = \frac{2C}{q^{1+\beta}}.\] In particular, for all $\beta>0$,
\[{\rm Leb}(S_\beta) = {\rm Leb}\left(\limsup_{q\to +\infty} S_{\beta,q}\right) \leq \limsup_{q\to +\infty}\sum_{r\geq q}\frac{2C}{r^{1+\beta}}= 0.\]

\section{The proof of Lemma \ref{lemma-qctwist}}
\label{sec-lem-qctwist}
To complete the proof of Lemma \ref{lemma-qctwist}, we will now construct a $K$-quasiconformal map
\[\Psi\colon\E(f^{\circ q})\to {\cal E}_\sigma \equiv E'\]
which sends the class of $\RZ$ in $\E(f^{\circ q})$ to the class of $\sigma\bbR/\sigma\bbZ$ in ${\cal E}_\sigma$. We will also show that
$\log K\leq 5D_f$.
The result then follows from Lemma \ref{lem-quasi-dist}.

The homeomorphism $\psi_j\colon C_j^-\to C_j^-$ given by
\[\psi_j(z) \dfn \xi_j(z)-s_{j+1}+r_j\]
fixes $r_j\in C_j^-$. Let $\tilde \psi_j\colon\bbR\to \bbR$ be the lift of $\psi_j\colon C_j^-\to C_j^-$ which fixes $\tilde r_j$ and let $\tilde \Psi_j\colon \overline S_j\to \overline S_j$ be its extension to $\overline S_j$ defined by
\[\tilde \Psi_j(x+\i y) \dfn \frac{y}{m_j} (x+\i m_j) + \left(1-\frac{y}{m_j}\right) \tilde \psi_j(x).\]
The homeomorphism $\tilde \Psi_j\colon\overline S_j\to \overline S_j$ restricts to the identity on $\bbR+\i m_j$ and descends to a homeomorphism $\Psi_j\colon\overline B_j\to \overline B_j$. By construction, the following diagram commutes:
\[
\xymatrix{
C_j^- \ar[r]^{\Psi_j} \ar[d]_{\xi_j} & C_j^-\ar[d]^{z\mapsto z-r_j+s_{j+1}} \\
C_{j+1}^+ \ar[r]_{\Psi_{j+1}} & C_{j+1}^+.
}\]
So, the collection of homeomorphisms $\Psi_j\colon\overline B_j\to \overline B_j$ induces a global homeomorphism $\Psi\colon \E(f^{\circ q})\to E'$ (see Fig. \ref{fig-Psi}). Since $ \tilde \Psi_j$ fixes $\tilde r_j$ and $\tilde s_j$, the homeomorphism $\Psi$ sends the homotopy class of $\delta$ in $\E(f^{\circ q})$ to the homotopy class of $\delta'$ in $E'$.

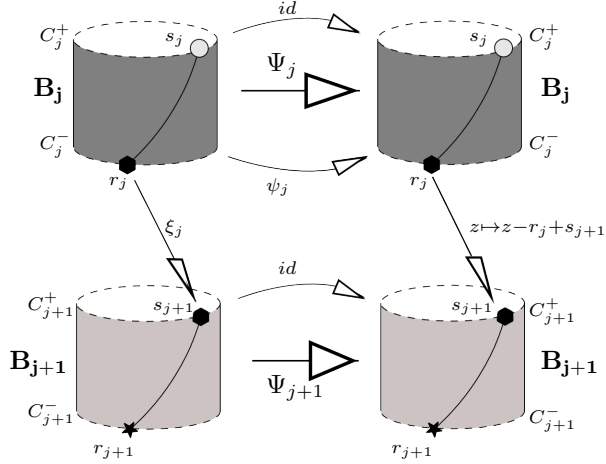
\begin{figure}[htbp]
\centerline{
\begin{picture}(0,0)%
\scalebox{.75}{\includegraphics{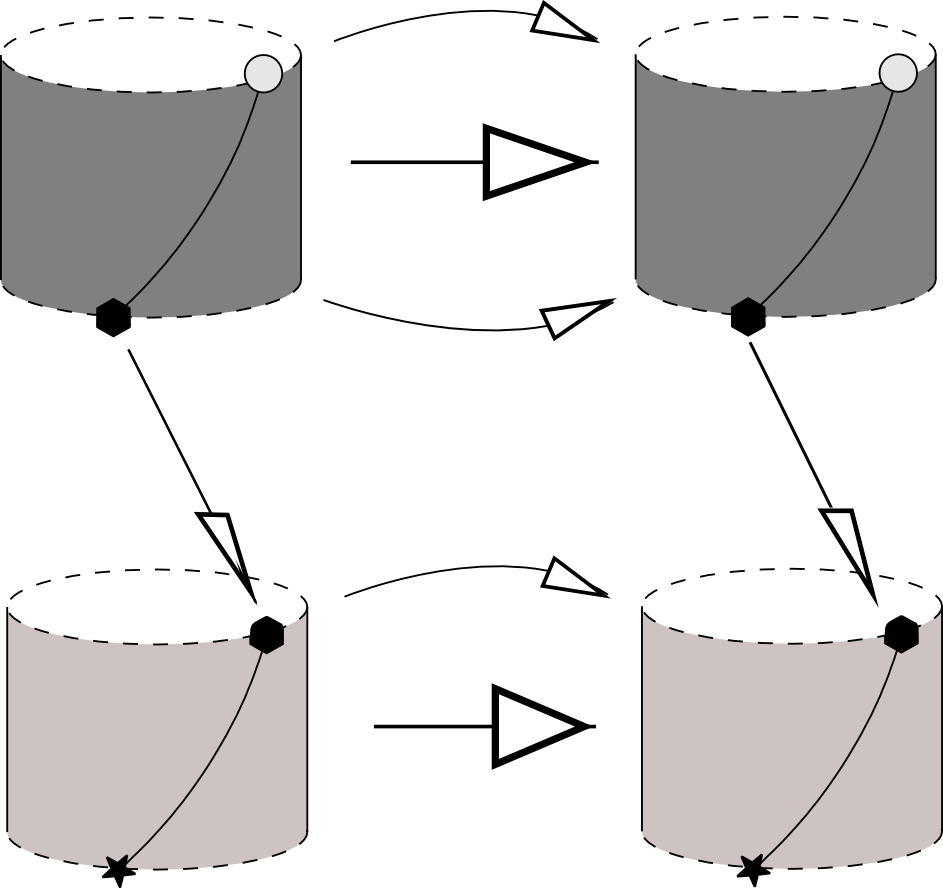}}
\end{picture}%
\setlength{\unitlength}{2960sp}%
\begin{picture}(7288,3676)(1189,-3368)
\put(1482,-1250){$\scriptstyle r_j$}
\put(1950,-50){$\scriptstyle s_j$}
\put(4000,-1250){$\scriptstyle r_j$}
\put(4500,-50){$\scriptstyle s_j$}
\put(1350,-3500){$\scriptstyle r_{j+1}$}
\put(1850,-2300){$\scriptstyle s_{j+1}$}
\put(5100,-511){$\mathbf{ B_j}$}
\put(5000,-950){$\scriptstyle C_j^{-}$}
\put(5000,-50){$\scriptstyle C_j^{+}$}
\put(850,-511){$\mathbf{ B_j}$}
\put(900,-950){$\scriptstyle C_j^{-}$}
\put(900,-50){$\scriptstyle C_j^{+}$}
\put(650,-2800){$\mathbf{ B_{j+1}}$}
\put(800,-2300){$\scriptstyle C_{j+1}^{+}$}
\put(800,-3200){$\scriptstyle C_{j+1}^{-}$}
\put(3850,-3500){$\scriptstyle r_{j+1}$}
\put(4350,-2300){$\scriptstyle s_{j+1}$}
\put(5100,-2800){$\mathbf{ B_{j+1}}$}
\put(5000,-2350){$\scriptstyle C_{j+1}^{+}$}
\put(5000,-3250){$\scriptstyle C_{j+1}^{-}$}
\put(2800,-1300){$\scriptstyle \psi_j$}
\put(2800,-3000){$ \Psi_{j+1}$}
\put(2800,-300){$ \Psi_{j}$}
\put(2900,-2000){$\scriptstyle id$}
\put(2900,  180){$\scriptstyle id$}
\put(1950,-1650){$\scriptstyle \xi_j$}
\put(4500,-1650){$\scriptstyle z \mapsto z-r_j+s_{j+1}$}
\end{picture}
}
\caption{ The maps $\Psi_{j}$ and $\Psi_{j+1}$ \label{fig:Psij}}
\label{fig-Psi}
\end{figure}

Now, it suffices to prove that the homeomorphism $\Psi\colon \E(f^{\circ q})\to E'$ is $e^{5D_f}$-quasiconformal.

The images of the curves $C_j^\pm$ in $\E(f^{\circ q})$ are analytic (because the glueing map $\xi_j$ is analytic), therefore quasiconformally removable. So, it is enough to prove that each $\Psi_j\colon B_j\to B_j$ is $e^{5D_f}$-quasiconformal. Equivalently, we must prove that
\[\left\|\frac{\partial  \tilde \Psi_j/\partial \bar z}{\partial  \tilde \Psi_j/\partial z}\right\|_\infty\leq k<1\quad\text{with}\quad \distance_\bbD(0,k)<5D_f ,\]
where $\distance_\bbD$ is the hyperbolic distance within the unit disk.

For readibility, we drop the index $j$ in the following computation:
\begin{align*}
\frac{\partial  \tilde \Psi/\partial \bar z}{\partial \tilde  \Psi/\partial z} (x+\i y)&=\frac{\partial \tilde  \Psi/\partial x+\i \partial  \tilde \Psi/\partial y}{\partial   \tilde \Psi/\partial x-\i \partial   \tilde \Psi/\partial y}(x+\i y)\\
&= \frac{\left(1-\frac{y}{m}\right)\cdot \bigl(\tilde \psi'(x) -1\bigr)-\frac{\i}{m}\bigl(\tilde \psi(x)-x\bigr)}{2+\left(1-\frac{y}{m}\right)\cdot \bigl(\tilde \psi'(x) -1\bigr)+\frac{\i}{m}\bigl(\tilde \psi(x)-x\bigr)}.
\end{align*}
This last quantity is of the form $(a-1)/(\bar a+1)$ with
\[\Re(a)=1+\left(1-\frac{y}{m}\right)\cdot \bigl(\tilde \psi'(x) -1\bigr)\quad\text{and}\quad
\Im(a) = \frac{\tilde \psi(x)-x}{m}.\]
Note that $\displaystyle \left|\frac{a-1}{\bar a+1}\right|=\left|\frac{a-1}{a+1}\right|$
and the Möbius transformation $a\mapsto \displaystyle \frac{a-1}{a+1}$ sends the right half-plane into the unit disk. So, it is enough to show that $a$ belongs to the right half-plane $\Set{z\in \bbC|\Re(z)>0}$ and that the hyperbolic distance within this half-plane between $1$ and $a$ is at most $5D_f$.

This hyperbolic distance is bounded from above by $\bigl|\Im(a)\bigr| + \bigl|\log \Re(a)\bigr|$.
Since $\tilde\psi\colon\bbR\to \bbR$ is an increasing diffeomorphism which fixes $r+\bbZ\in\bbR$, we have that $\tilde\psi'(x)>0$ and
 $\bigl|\tilde \psi(x)-x\bigr|<1$. In addition, $0<1-y/m<1$,  and so,
\[0<\min_\bbR \tilde \psi' \leq \Re(a)\leq \max_\bbR \tilde \psi' \quad \text{and}\quad \bigl|\Im(a)\bigr|\leq \frac{1}{m} = \frac{|\log \rho|}{\pi}\leq |\log \rho|\leq D_f\]
(the last inequality is given by Lemma \ref{th_as-Denjoy}).
The average of $\tilde \psi'$ on $[0,1]$ is equal to $\tilde \psi(1)-\tilde \psi(0)=1$. So, $\tilde \psi'$ takes the value $1$ and
\[{-\distortion_\bbR (\xi)} = {-\distortion_\bbR(\tilde\psi)} <\log \min_\bbR  (\tilde \psi') \leq 0\leq \log \max_\bbR (\tilde \psi') < \distortion_\bbR(\tilde\psi) = \distortion_\bbR(\xi).\]

Now, $dis_{\mathbb R} \xi$ is at most $4 D_f$ due to Lemma \ref{lemma-disxi}.
This gives the estimate on $\Re a$:
\[
\exp (-4D_f) \le \Re a \le \exp(4D_f),
\]
and thus the required estimate on the distance between $a$ and $1$.

\end{document}